\documentclass[10pt,reqno]{amsart} 
\usepackage{amssymb}
\usepackage{amsmath}
\usepackage{bm}
\usepackage{amsthm}
\usepackage{epsfig}
\usepackage{mathrsfs}
\usepackage{graphicx}
\usepackage{ textcomp }
\usepackage[dvipsnames]{xcolor}
\usepackage[bookmarks=true, colorlinks=true]{hyperref}
\hypersetup{urlcolor=blue,citecolor=red,linkcolor=black}
\usepackage{cleveref}
\usepackage{soul}
\usepackage{ulem}
\usepackage{tikz}
\usepackage{float}
\usepackage{enumerate}
\usepackage{enumitem}
\usepackage{comment}
\usepackage{tikz}

\def\R{\mathbb{R}}
\def\N{\mathbb{N}}

\def\pa{\partial}

\DeclareMathOperator*{\esssup}{ess\,sup}

\DeclareMathOperator\supp{supp}

\newtheorem{theorem}{Theorem}[section]
\newtheorem{lemma}[theorem]{Lemma}
\newtheorem{proposition}[theorem]{Proposition}
\newtheorem{definition}[theorem]{Definition}
\newtheorem{remark}[theorem]{Remark}
\newtheorem{corollary}[theorem]{Corollary}

\makeatletter
\newcommand{\myitem}[1]{%
	\item[#1]\protected@edef\@currentlabel{#1}%
}
\makeatother

\date{\today}

\title{Mean Field Limit for Congestion Dynamics in one dimension}
\author{Inwon Kim \and Antoine Mellet \and Jeremy Sheung-Him Wu}
\address{I. Kim, J. S.-H. Wu -- Mathematical Sciences Building, University of California, Los Angeles, CA 90095, United States.}

\email{ikim@math.ucla.edu}
\email{jeremywu@math.ucla.edu}

\address{A. Mellet -- Department of Mathematics, University of Maryland, College Park, MD 20742-4015, United States.}
\email{mellet@math.umd.edu}

\thanks{I. Kim was partially supported by NSF Grant DMS-2153254. \\
A. Mellet was partially supported by NSF Grant DMS-2009236 and DMS-2307342.
}

\begin{document}
	\maketitle
	\begin{abstract}
This paper addresses congested transport, which can be described, at macroscopic scales, by a continuity equation with a pressure variable generated from the hard-congestion constraint (maximum value of the density). The main goal of the paper is to show that, in one spatial dimension, this continuum PDE can be derived as the mean-field limit of a system of ordinary differential equations that describes  the motion of a large number of particles constrained to stay at some finite distance from each others. To show that these two models describe the same dynamics at different scale, we will rely on both the Eulerian and Lagrangian points of view and use two different approximations for the density and pressure variables in the continuum limit.
	\end{abstract}

	\section{Introduction}
	\label{sec:intro}
In this paper we consider a simple one dimensional model for congestion dynamics of $N$ particles moving along the real line. Their desired velocity field is given by the gradient of a smooth potential $\phi:\R\to\R$, but the actual motion of the particles must balance this velocity with a hard congestion constraint, which requires that the particles remain at some distance $2r>0$ from each other at all time.
As in \cite{MRSV11}, we assume that the actual velocity of each particle is then the projection of the desired velocity onto the set of feasible velocities so that the trajectories preserve the non-overlapping constraint~\eqref{non_overlap}. Denoting by $x_i(t)$ the center of particle positions, the equation is
\begin{equation}
		\label{eq:EL_r}
			\dot{x}_i(t) = -\phi'(x_i(t)) - N(\lambda_i(t) - \lambda_{i-1}(t)), 	\qquad \forall t\in[0,T],\qquad \forall i=1,\dots, N.
	\end{equation}
Here the $\lambda_i(t)$ are \textit{Lagrange multipliers} introduced to enforce the non-overlapping condition which reads:
	\begin{equation}\label{non_overlap}
		x_{i+1}(t) - x_i(t) \geq  2r,   \qquad \forall t\in[0,T],\qquad \forall i=1,\dots, N-1.
	\end{equation}
 The equation for $\lambda_i(t)$ is then given by (see \cite{MRSV11}):
	\begin{equation}
		\label{eq:slack_r}
		 \lambda_i\ge 0, \, \qquad \lambda_i(x_{i+1} - x_i - 2r)=0,\qquad \forall t\in[0,T] ,  \qquad \forall i=1,\dots, N-1.
	\end{equation}
Since the first and last particles are unconstrained from one direction, we have
	\begin{equation}\label{eq:lambdabc}
		\lambda_0(t) = \lambda_N(t) = 0, \quad \forall t\in[0,T].
\end{equation}
Finally, we supplement~\eqref{eq:EL_r} with initial data $x_i(0):=x_i^0$ where the $x_i^0$ satisfy the non-overlapping condition~\eqref{non_overlap}.
The derivation of~\eqref{eq:EL_r} and in particular the role of condition \eqref{eq:slack_r} will be made clear in~\Cref{sec:micro}. We refer to~\cite{MRSV11} (see also~\cite{M18}) for a further interpretation of this system, its generalization in higher dimension, and its connection to sweeping processes. 
We will also provide a proof of the existence and uniqueness of solutions (see Proposition \ref{prop:exist}), using a slightly different discrete-time scheme than the one used in \cite{MRSV11}, with an emphasis on the macroscopic limit.

\medskip

When the number of particles $N$ is very large and the radius $r$ scales like $1/N$ (so that the total volume occupied by the particles stays of order $1$), we would like to describe the dynamics encoded by the system \eqref{eq:EL_r}--\eqref{eq:lambdabc} 
by some effective macroscopic equation that 
does not track the position of individual particles, but instead models the evolution of a density distribution function $\rho(t,x)$.
The description of hard congestion at such macroscopic scale has been a popular topic of research in recent decades, with applications in particular in the modeling of crowd motion \cite{MRS10,MV11,MRSV11,AKY14,M18, KPW19, S_survey,LMS20}, tumor growth \cite{PQV,KP18, MPQ,DS} and aggregation phenomena \cite{CKY,KMW23,KMW24}. At the macroscopic level, the non-overlapping constraint is usually expressed by prescribing a maximal value for the density of particles, leading to the following equation:  
	\begin{equation}
		\label{eq:macro}
		\partial_t \rho - \pa_x(\rho (\pa_x \phi+\pa_x p))=0,  \qquad p\geq 0 \quad 0\leq \rho \le 1, \quad p(1-\rho) = 0, \quad \text{in }(0,T)\times \R,
			\end{equation}
where the pressure function $p(t,x)$, like the $\lambda_i(t)$ in \eqref{eq:EL_r}, is a Lagrange multiplier which enforces the  maximal density (or ``incompressibility") constraint $\rho \le 1$.
\medskip

In many of the references mentioned above (see for instance \cite{AKY14,KP18, PQV,CKY}), the equation \eqref{eq:macro} is derived as the limit of a nonlinear diffusion model (the so-called incompressible limit of the porous media equation). 
In \cite{MRSV11} it is introduced directly by defining the velocity field $-\pa_x \phi+\pa_x p$ as the projection of the spontaneous velocity $-\pa_x \phi$ onto the set of feasible velocity fields, which preserves the constraint $\rho \leq 1$. 
This interpretation of \eqref{eq:macro} shows that the microscopic and macroscopic models rely on the same underlying mechanism:
hence one might expect that when the number of particles $N$ is very large and the size of the particles $r$ satisfies $r\sim \frac 1 N$, 
the solutions of the system  ~\eqref{eq:EL_r}--\eqref{eq:lambdabc} can be effectively described by a density distribution solving  \eqref{eq:macro}. 

It is important to point out that such convergence does not appear to be true, at least with any straightforward extension of the microscopic system, in dimensions higher than $1$: this is because the dynamics of the microscopic model is much richer than that of the macroscopic one in higher dimensions.
More precisely the cluster shape of packed particles depend sensitively on the particle shape, while the macroscopic limit loses this information~\cite[Section 5]{MRSV11}. 
We refer to section 4 of the survey article~\cite{KMW24} for further discussions on this issue as well as a review of different approximations and partial results in this direction. 
In contrast, our goal in this article is to stay in one space dimension and make the derivation of \eqref{eq:macro}
from \eqref{eq:EL_r}--\eqref{eq:lambdabc} rigorous. We will show that the empirical distribution
\begin{equation}\label{eq:empirical} 
\rho_N(t,x) := \frac 1 N \sum_{i=1}^N \delta(x-x_i(t)),
\end{equation}
where $x_1(t),\dots ,x_N(t)$ solve \eqref{eq:EL_r}--\eqref{eq:lambdabc} for some appropriate initial data, converges to a solution of \eqref{eq:macro} when $N\to\infty$ with $2r=1/N$.

\medskip

An interesting challenge, even in this one dimensional framework, is that the empirical distribution $\rho_N(t,x)$ is ill-suited for a derivation of the 
 pointwise macroscopic constraint $\rho\leq1$. 
This will lead us to introduce another approximation of the macroscopic density:
\begin{equation}
\label{eq:histr}
\tilde{\rho}_N(t,x) = \sum_{i=0}^{N-1}\frac{1}{N(x_{i+1} (t)- x_{i}(t))}\chi_{[x_{i+1}(t),x_{i}(t))}(x),
\end{equation} 
for which the constraint $\tilde{\rho}_N\leq 1$ is equivalent to \eqref{non_overlap} (note that in order to keep the mass of  $\tilde{\rho}$ 
equal to $1$, we added a point $x_0(t):=x_1(t)-\frac 2 N$). Lemma \ref{lem:dist_emp_q} shows that $\rho_N$ and $\tilde\rho_N$ are vanishingly close in some weak sense.
	\medskip

The main idea to establish a rigorous link between \eqref{eq:EL_r}--\eqref{eq:lambdabc} and \eqref{eq:macro},
is to proceed as in \cite{FKS24} and interpret the particle index $i$ as a discrete Lagrangian variable by defining 
$$ \Delta s = \frac 1 N, \qquad s_i = i\Delta s, \qquad i=1,\dots, N,$$
and introducing functions $X_N(t,s)$ and $\Lambda_N(t,s)$ such that
$$ X_N(t,s_i)=x_i(t), \qquad \Lambda_N(t,s_i) = \lambda_i(t),\qquad i=1,\dots, N.$$
When $2r: =\frac 1 N=\Delta s$, the formal limit $N\to\infty$ of \eqref{eq:EL_r}--\eqref{eq:lambdabc} yields
	\begin{equation}
		\label{ode:1}
		\left\{
		\begin{array}{ll}
			\partial_t X = -\phi'(X) -\partial_s \Lambda &\text{in }[0,T]\times[0,1] 	\\
			\Lambda(t,0) = \Lambda(t,1) = 0 	&\text{in }[0,T]
		\end{array}
		\right.
	\end{equation}
	together with the  condition	
	\begin{equation}
		\label{ode:2}
		\Lambda\ge 0,\,\partial_s X \ge 1,\text{ and } \Lambda(1 - \partial_s X) = 0, \quad \text{for almost every }(t,s)\in[0,T]\times[0,1].
	\end{equation}
	
 Theorem \ref{thm:q1} will make the above computations precise, providing a first rigorous derivation of~\eqref{eq:macro} from the Lagrangian framework of particle dynamics.  For a quick illustration, we first present here a formal derivation of the system  \eqref{ode:1}--\eqref{ode:2} from \eqref{eq:macro}. For a given density $\rho(t,x)$, the \textit{pseudo-inverse} of its cumulative distribution function (or the \textit{quantile function}) is given by
	\begin{equation}\label{eq:dist}
		X(t,s) = X_{\rho(t)}(s) := \inf\left\{x\in \R \, \left| \, \int_{-\infty}^x \rho(t,y) \right. \mathrm{d}y \ge s\right\}, \quad (t,s)\in[0,T]\times [0,1],
	\end{equation}
we can  think of $X(t,s)$ as the `location of the particle at time $t$ corresponding to the $s$--th percentile.'
For later purposes, we note that this relation can be inverted using the push-forward map (c.f.~\cite[Proposition 2.2]{S15}):
	\begin{equation}
		\label{eq:push_forward}
		\rho(t,\cdot) := X(t,\cdot)_\# (\mathrm{d}s ),
	\end{equation}
where $\mathrm{d}s$ denotes  the restriction of the Lebesgue measure on $(0,1)$. 
	When $\rho$ is bounded, \eqref{eq:dist} implies
	\begin{equation}
		\label{eq:x_rho}
		\int_{-\infty}^{X(t,s)}\rho(t,y)\, \mathrm{d}y = s.
	\end{equation}
which gives (differentiating with respect to $s$):
		\begin{equation}
			\label{eq:x_rho_omega}
			\rho(t,X(t,s))\, \partial_sX = 1.
		\end{equation}
We now claim that $X$ solves \eqref{ode:1}--\eqref{ode:2} if $\rho$ solves \eqref{eq:macro}. To see this, first note that the saturation condition $p(1-\rho) = 0$ and $0\leq \rho \le 1$ can be written, in terms of $p$ and $X$, as
		\[
		p(t,X(t,s))(\partial_s X(t,s)-1)=0, \quad \text{and}\quad \partial_s X \ge 1.
		\]
which is \eqref{ode:2} if we define  $\Lambda$ by
	\begin{equation}
		\label{def:pressure}
		\Lambda(t,s) := p(t,X(t,s)), \quad \forall (t,s)\in[0,T]\times [0,1].
	\end{equation}
Furthermore, the time derivative of~\eqref{eq:x_rho} leads to
	\[
	(\partial_t X) \rho(t,X(t,s)) + \int_{-\infty}^{X(t,s)}\partial_t\rho(t,y)\, \mathrm{d}y = 0.
	\]
	Since $\rho$ solves~\eqref{eq:macro}, we get
	\begin{align*}
		0 	= (\partial_t X)\rho(t,X(t,s)) + \int_{-\infty}^{X(t,s)}\partial_y(\rho\partial_y(\phi + p))\, \mathrm{d}y 	
		 = \rho(t,X(t,s))\, (\phi'(X(t,s)) + \partial_t X + \partial_x p(t,X(t,s))).
	\end{align*}
	Therefore, on the support of $\rho$, we obtain
	\[
	\partial_t X  = - \phi'(X(t,s))- \partial_x p(t,X(t,s)).
	\]
Observe that the saturation condition $p(1-\rho) = 0$ formally implies $\partial_xp = \rho\partial_xp$, and thus
	\[
	\partial_s \Lambda(t,s) = \partial_x p(t,X(t,s)) \, \partial_s X = \partial_x p(t,X(t,s)).
	\]
	Hence the above equation for $X$ yields \eqref{ode:1}, and we have shown the claim. The homogeneous boundary conditions for $\Lambda(t,s)$ at $s=0,1$ in~\eqref{ode:1} correspond to the absence of pressure exerted on the left and right-most particles.
	\medskip
	
These computations can be justified rigorously at least when the potential $\phi:\R\to \R$ is regular enough. Throughout the paper we assume that $\phi$ is $C^2$ and satisfies

\begin{equation}\label{phi}
	c_2:=\sup_{x\in \R}|\phi''|(x) <+\infty\,\text{ and }\phi(x) \ge c_0(1+|x|^2), \text{ for some }c_0,c_2>0.
\end{equation}

	\begin{theorem} 
		\label{thm:wp_micro}
		Suppose that $\phi$ satisfies \eqref{phi}.
		For a given $X^0:[0,1]\to \R$ satisfying 
	\begin{equation}\label{x^0} 
 X^0(s_2)-X^0(s_1) \ge s_2-s_1 \quad \mbox{ for all }  0\leq s_1<s_2\leq 1,
\end{equation}
there is a {\bf unique} pair  $X\in H^1(0,T; \, L^2(0,1))$ and $\Lambda \in L^2(0,T; \, H^1(0,1))$ solving ~\eqref{ode:1}--\eqref{ode:2} with $X(0,s) = X^0(s)$.
		
		Furthermore, the pair of functions $(\rho,p)$ defined from $(X,\Lambda)$ by~\eqref{eq:push_forward} and~\eqref{def:pressure} is the {\bf unique} weak solution 
			to~\eqref{eq:macro} with initial condition $\rho(0,\cdot) = \rho^0:=X^0(\cdot)_\# (\mathrm{d}s ).$
	\end{theorem}
	\begin{remark}
	We will also prove  (see~\Cref{lem:omega_i_exp}) that  the solution of \eqref{ode:1}--\eqref{ode:2}  satisfies
			$$1\leq  \pa_s X(t,s) \leq \pa_s X^0(s) e^{\|\phi''\|_{L^\infty} t},\qquad \forall t\geq 0,$$
			where this inequality should be understood in the sense of Radon measures (both $X(t,\cdot) $ and $X^0$ are BV functions). This implies in particular that the function $s\mapsto X(t,s)$ cannot have a jump discontinuity at $s_0$ unless the initial condition already had a  jump discontinuity at $s_0$.
			Such discontinuities correspond to intervals where the density $\rho$ vanishes so this tells us that the number of connected components of $\supp \rho(t)$ cannot increase in time. 
	\end{remark}
	Note that  \eqref{x^0} implies in particular that $X^0$ is monotone increasing and satisfies $(X^0)'(s)\geq 1$, which is the continuous expression of the non-overlapping condition in Lagrangian coordinates. 
The uniqueness results included in Theorem  \ref{thm:wp_micro} means that there is a one-to-one correspondance between the Lagrangian formulation \eqref{ode:1}--\eqref{ode:2} and the Eulerian one \eqref{eq:macro} for the macroscopic description of congested dynamic in dimension $1$.

\medskip

We recall here the classical definition (we refer to~\Cref{sec:prelim} for the meaning of $AC([0,T];\mathscr{P}_2)$ and other preliminaries) of weak solution for \eqref{eq:macro}: 
\begin{definition}[Weak solutions to~\eqref{eq:macro}]
	\label{def:wk_macro}
	For a given $T>0$ and $\phi$ satisfying \eqref{phi}, the pair $(\rho, \, p) \in  AC([0,T];\mathscr{P}_2)\times  L^2(0,T;\, H^1(\R))$ is a \textit{weak solution} to~\eqref{eq:macro} if
	\begin{equation}\label{wk_sat}  
	p(1-\rho)=0,\quad p\ge0, \mbox{ and } 0\le \rho \le 1 \mbox{ a.e. } (t,x)\in[0,T]\times \R.
	\end{equation}	
		and
		\begin{equation}
			\label{eq:wk_test}
			\partial_t\rho = \partial_x(\rho\partial_x\phi) + \partial_{xx}p, \hbox{ in } \mathcal{D}'((0,T)\times \R).
		\end{equation}
Note that, since \eqref{wk_sat} implies that $\partial_xp = \rho\partial_xp$ a.e., the pair $(\rho,p)$ also satisfies \eqref{eq:macro}.
\end{definition}
	
The existence and uniqueness of weak solutions in the sense of Definition \ref{def:wk_macro} is known (see~\cite{DMM16} for the uniqueness result) when the initial condition satisfies, for instance:
	\begin{equation}\label{rho^0}
	0\leq \rho^0	\leq 1, \qquad \int_\R \rho^0(x)\,\mathrm{d}x =1, \qquad\mbox{and } \supp \rho^0 \mbox{ is compact}.
	\end{equation}

\medskip

The existence of $(X,\Lambda)$ in Theorem  \ref{thm:wp_micro}  will be a consequence of Theorem \ref{thm:q1} below which proves the convergence of the system \eqref{eq:EL_r}--\eqref{eq:lambdabc} to \eqref{ode:1}--\eqref{ode:2} when $N\to\infty$.
For the sake of completeness, we will thus give a proof of the existence of a solution to \eqref{eq:EL_r}--\eqref{eq:lambdabc}:
	\begin{proposition}[Well-posedness of the discrete system]
		\label{prop:exist}
		For any $r>0$ and $N\in\N$, given $ \mathbf{X}^0=(x_1^0,\dots x_N^0)\in \R^N$ such that
		\begin{equation}\label{eq:x_0N}
		x^0_{i+1} - x^0_i \geq 2r, \qquad \forall i=1, \dots, N-1 ,
		\end{equation} 
		  there exist unique $ \mathbf{X}(t)=(x_1(t),\dots,x_N(t)) \in H^1(0,T;\, \R^N)$ and $ \mathbf{\Lambda}(t)=(\lambda_0(t),\lambda_1(t),\dots,\lambda_N(t))   \in L^2(0,T;\,\R^{N+1})$ solutions of  ~\eqref{eq:EL_r}--\eqref{eq:lambdabc}.
	\end{proposition}
	\Cref{prop:exist} is not new in the literature, it can be considered as a particular case of Theorem 2.10 and Remark 2.11 in~\cite{MV11}. However, we provide an explicit construction to clarify estimates on $x_i(t)$ and $\lambda_i(t)$ which are essential in passing to the limit $N\to \infty$.
\medskip
	
We can now make precise the construction of the functions $X_N(t,s)$ and $\Lambda_N(t,s)$. For a given $N\in \N$ with $2r=\frac 1 N$, consider a solution $ (\mathbf{X}(t), \mathbf{\Lambda}(t))$ of \eqref{eq:EL_r}--\eqref{eq:lambdabc} given by \Cref{prop:exist}. We define  the piecewise {\bf constant} function $X_N$ by
\begin{equation}\label{eq:XN}
X_N(t,s) = x_i(t), 	\quad s\in\left(\frac{i-1}{N}, \frac{i}{N}\right], \quad i=1,\dots, N,
\end{equation}
and the continuous piecewise {\bf linear} function $\tilde \Lambda_N$ such that $ \tilde\Lambda_N(t,s_i) = \lambda_i(t)$ for all $i=0,\dots, N$, that is
\begin{equation}\label{eq:LambdaN}
	 \tilde{\Lambda}_N(t,s):= N(\lambda_i(t) - \lambda_{i-1}(t))\left(s - \frac{i-1}{N}\right)+\lambda_{i-1}(t), \quad s\in\left[\frac{i-1}{N}, \frac{i}{N}\right], \quad i=1,\dots, N.
\end{equation}
Note that $X_N(t,s)$ is given by \eqref{eq:dist} corresponding to the empirical distribution $\rho_N(t,x)$ defined by \eqref{eq:empirical}.
Our motivation for the usage of $ \tilde{\Lambda}_N$ over a piecewise-constant interpolation is that, besides the obvious advantage of regularity,
  the pair
 $(X_N, \tilde{\Lambda}_N)$ solves \eqref{ode:1}, which makes the limiting process a little easier.
	
\medskip

Our first convergence theorem is the following:
		\begin{theorem}[Convergence in Lagrangian coordinates]
		\label{thm:q1}
For all $N\in \mathbb{N}$, 
		  let $(X_N(t), \tilde{\Lambda}_N(t))$ be the interpolations as given above by \eqref{eq:XN}--\eqref{eq:LambdaN}. Assume further that the initial condition $X_N(0,s)=X_N^0(s)$ satisfies
		  \begin{align}
		 	\label{x^0_phi}
		\bar \phi_N:= \sup_{N\in\mathbb{N}} \int_0^1\phi( X_N^0(s))\, \mathrm{d}s &<+\infty, 	\\
	\label{bv_endpts}
 \sup_{N\in\mathbb{N}} |X_N^0(1) -X_N^0(0)| &<+\infty,\text{ and} 	\\
	\label{eq:x^0_conv}
			\|X_N^0 - X^0\|_{L_s^1(0,1)}&\overset{N\to \infty}{\to} 0,
\end{align}
for some $X^0:[0,1]\to \R$ satisfying~\eqref{x^0}.
Then the following convergences hold:
$$			X_N \overset{N\to\infty}{\to} X\quad  \mbox{ strongly in $C([0,T]; \, L^p(0,1))$ for any $p\in[1,2)$}$$
and
$$
\tilde {\Lambda}_N \overset{N\to\infty}{\rightharpoonup} \Lambda\quad \mbox{ weakly in $L^2(0,T; H^1(0,1))$}, 
 $$
 where $(X,\Lambda)$ 
 is the  unique solution of \eqref{ode:1}--\eqref{ode:2}  provided by Theorem \ref{thm:wp_micro} with initial condition $X^0$.	 
			 	\end{theorem}
In Eulerian coordinates, \eqref{x^0_phi} reads $\sup_{N\in\N}\int_\R \phi(x) \, \rho_N^0(\mathrm{d}x)<+\infty$ which means that the potential energy in the system is uniformly bounded in $N$. It is thus essential to derive several bounds on $X_N$ and $\tilde \Lambda_N$ (see Proposition \ref{prop:sigma_cpct_1}). The bound on the density support \eqref{bv_endpts} will imply a BV bound that will be crucial in passing to the limit in the nonlinear term $\phi'(X_N)$ (see Proposition \ref{prop:BV_est}).
				
				


\medskip

Our last theorem will  show the convergence result for the density distribution:
\begin{theorem}[Convergence in Eulerian coordinates]
	\label{thm:q2}
	Under the assumptions of Theorem~\ref{thm:q1}, let  $\rho_N(t)$ be the empirical measure associated to $\mathbf{X}(t)=(x_1(t),\dots, x_N(t))$ by \eqref{eq:empirical} and let $p_N$ be  defined by
	\begin{equation}
	p_N(t,x) 	:= \sum_{i=0}^{N-1}\lambda_i(t)\chi_{[x_i(t),x_{i+1}(t))}(x) . \label{eq:disc_press_const}
\end{equation}
	Then, as $N\to\infty$, we have
	 $\rho_N(t)\to \rho(t)$ narrowly uniformly in $t\in[0,T]$  and
  $p_N\rightharpoonup p $  in $L^2((0,T)\times \R)$,
	where $(\rho,p)$ is the unique weak solution  to~\eqref{eq:macro}  with initial condition $ \rho^0:=X^0(\cdot)_\# (\mathrm{d}s )$.
\end{theorem}

\medskip

	\begin{figure}[H]
		\centering
		\begin{tikzpicture}
			\node[draw, align = center] at (-4,2) {$(\mathbf{X}_N,\mathbf{\Lambda}_N)$ solve\\\eqref{eq:EL_r} and~\eqref{eq:slack_r}};
			\node[draw, align = center] at (4,2) {$(X,\Lambda)$ solve\\\eqref{ode:1} and~\eqref{ode:2}};
			\node[draw, align = center] at (-4,-2) {$\rho_N$, $\tilde \rho_N$ and $p_N$ \\ solve \eqref{wk_sat_disc}  and \eqref{eq:wk_test_cty_disc}};
			\node[draw, align = center] at (4,-2) {$(\rho,p)$ solve~\eqref{eq:macro}};
			\draw[->] (-2.5,2) -- (2.7,2) node[pos=0.5, above, align = center] {$N\to \infty$\\\Cref{thm:q1}};
			\draw[->] (-2.2,-2) -- (2.6,-2) node[pos=0.5, below, align = center] {$N\to \infty$\\\Cref{thm:q2}};
			\draw[->] (-4, 1.3) -- (-4,-1.3) node[pos=0.5, left, align = center] {\Cref{prop:disc_wk_macro} };
			\draw[->] (4, 1.3) -- (4,-1.3) node[pos=0.5, right, align = center] {\Cref{thm:wp_micro}};
			\draw[->, color = Cerulean] (-2.5,1.5) -- (3.5,-1.5) node[pos=0.5, above right, align = center] {Eulerian};
			\draw[->, color = ForestGreen] (-3,1.4) -- (3,-1.6) node[pos=0.5, below left, align = center] {Lagrangian};
		\end{tikzpicture}
		\caption{Summary of the main results.}
		\label{fig:qs}
	\end{figure}
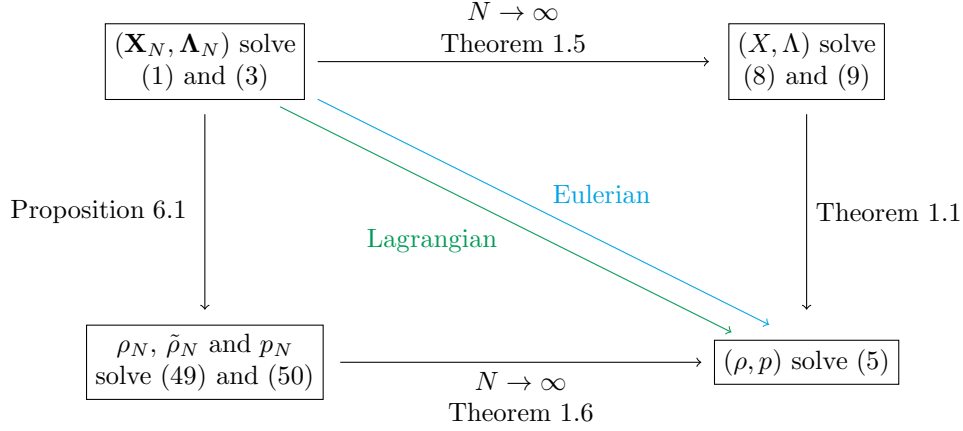

We have organised the paper in the following way:  \Cref{sec:micro} is dedicated to the proof of~\Cref{prop:exist} with an explicit construction using the JKO scheme~\cite{JKO98,S17}. \Cref{sec:N_to_inf}, the main section of the paper, settles crucial estimates that are uniform in $N$. Independent (but nevertheless crucial) to the $N\to \infty$ limit, \Cref{sec:uniqueness} establishes uniqueness of solutions to~\eqref{ode:1}--\eqref{ode:2}. These results are essential to the proofs of~\Cref{thm:wp_micro,thm:q1} (resp.~\Cref{thm:q2}) which are in~\Cref{sec:q1} (resp.~\Cref{sec:q2}).

\begin{remark}
	\label{rem:interaction}
We believe our technique also extends to non-local models of congestion. For example, if $W\in C_b^2$ is a symmetric interaction kernel, we can modify~\eqref{eq:EL_r}, \eqref{eq:macro}, and~\eqref{ode:1} to the following
\begin{align*}
	\dot{x}_i(t) &= -\frac{1}{N}\sum_{j=1}^NW'(x_i(t) - x_j(t)) - \phi'(x_i(t))-N(\lambda_i - \lambda_{i-1}),	\\
	\partial_t X &= - \int_0^1W'(X(t,s) - X(t,\sigma))\,\mathrm{d}\sigma - \phi'(X) - \partial_s \Lambda, \text{ and}	\\
	\partial_t\rho &= - \partial_x(\rho\partial_x(W*\rho + \phi + p)).
\end{align*}
\end{remark}

\Cref{sec:sampling} details a particular construction of the initial conditions for the particle system from the initial density $\rho^0$. \Cref{sec:prelim} records preliminary facts on Wasserstein distances used in the paper.

	\section{The microscopic model}
	\label{sec:micro}
	In this section we prove Proposition \ref{prop:exist}, which yields the existence of a unique solution to the microscopic model \eqref{eq:EL_r}--\eqref{eq:lambdabc}.
While this result holds for any $r>0$ and $N\in \N$, we will take $2r=\frac 1 N$ in the proof below to be consistent with the rest of the paper where this choice of $r$ is crucial.
In particular this allows a clear exposition to show how the estimates derived in this section will be used to pass to the limit $N\to\infty$ later. 
\medskip
	
We define the set of admissible particle positions  (when  $2r = \frac{1}{N}$) by
\[
\mathcal{K}_N := \left\{
\mathbf{X}=(x_1,\dots, x_N)\in \R^N \, \left| \, x_{i+1}- x_i \geq  \frac{1}{N}\right., \quad \forall i=1,\dots, N-1
\right\},
\]
and for $\mathbf{X} \in \mathcal{K}_N$, we define the set of admissible velocities as
\[
C(\mathbf{X}):= \left\{
\mathbf{V} = (v_1,\dots, v_N)\in \R^N \, \left| \, \text{if }x_{i+1} - x_i =\frac{1}{N}\right. \text{ for some }i=1,\dots, N-1, \text{ then }v_{i+1}\ge v_i
\right\}.
\]
Given $\mathbf{X}\in \mathcal{K}_N,$ the map $$\mathbb{P}_{C(\mathbf{X})}:\R^N \to C(\mathbf{X})$$ will denote the (Euclidean) projection onto $C(\mathbf{X})$.
	Our goal in this section is to construct solutions to the following system of ODEs:
	\begin{equation}
		\label{eq:micro}
		\dot{\mathbf{X}}(t) = \mathbb{P}_{C(\mathbf{X}(t))}(-\Phi'(\mathbf{X}(t))), \quad t\in(0,T), \quad  \mathbf{X}(0) = \mathbf{X}^0 \in \mathcal{K}_N,
	\end{equation}
and show that it is equivalent to \eqref{eq:EL_r}--\eqref{eq:lambdabc}
(we used the shorthand notation $\Phi'(\mathbf{X}) = (\phi'(x_1), \dots, \phi'(x_N))$). As mentioned before, this result is not new: the general version is available in~\cite{MV11}.
A constructive approach is taken here, which will yield some crucial estimates for passing to the limit $N\to \infty$ in~\Cref{sec:N_to_inf}.
	\medskip
	
We use the classical variational minimizing movement / JKO scheme~\cite{JKO98,AGS08,S17} to construct solutions to~\eqref{eq:micro} Given a time step $\tau>0$, 
	we iterate the following optimization problem:
	\begin{equation}
		\label{eq:JKO_micro}
		\mathbf{X}^0 \in \mathcal{K}_N, \quad \mathbf{X}^{k+1} \in \text{argmin}_{\mathbf{X}\in \mathcal{K}_N}\left\{
		\frac{1}{N}\sum_{i=1}^N \phi(x_i) + \frac{|\mathbf{X}-\mathbf{X}^k|^2}{2N\tau}
		\right\}, \quad k \in \N. 
	\end{equation}
	Owing to \eqref{phi} and the convexity of $\mathcal{K}_N$, there is a unique $\mathbf{X}^{k+1}$ for each $k\in\N$, if $\tau$ is sufficiently small.
	The Euler-Lagrange equation is given with the Lagrange multipliers $\mathbf{\Lambda}^{k+1} = (\lambda_1^{k+1}, \dots, \lambda_{N-1}^{k+1})\in \R^{N-1}$:
	\begin{equation}
		\label{eq:EL_micro}
		\left\{
		\begin{array}{rl}
			\frac{1}{N}\phi'(x_i^{k+1}) + \frac{1}{N}\left(\frac{x_i^{k+1}-x_i^k}{\tau}\right) + \lambda_i^{k+1} - \lambda_{i-1}^{k+1} = 0,  &\forall i=1,\dots, N,	\\
			\text{ where }\lambda_0^{k+1} :=0\text{ and } \lambda_N^{k+1} := 0, 
		\end{array}
		\right. \quad \forall k\geq 0.
	\end{equation}
	This is paired with the complementary slack conditions
	\begin{equation}
		\label{eq:slack_micro}
		\lambda_i^{k+1}\left(-x_{i+1}^{k+1} + x_i^{k+1} + \frac{1}{N}\right) = 0, \quad x_{i+1}^{k+1} - x_i^{k+1}\geq \frac{1}{N}, \quad \lambda_i^{k+1}\ge 0,\quad \forall i=1,\dots, N-1, \quad \forall k\geq 0 .
	\end{equation}
		Our aim is to construct a solution of \eqref{eq:micro} in the limit $\tau\to0$.
		 We begin with some immediate and important estimates from ~\eqref{eq:JKO_micro}.
	\begin{lemma}
		\label{lem:JKO_est}
		Let $(x_i^k)$ and $(\lambda_i^{k+1})$ be described as above with $\phi$ satisfying \eqref{phi}. For any $T=\tau K$ with $K\in \N$, the following estimates hold where $\bar{\phi}_N := \frac{1}{N}\sum_{i=1}^N\phi(x_i^0)$:
		\begin{align}
				\label{eq:utsnx}
				\sup_{k\geq 0}\left(\frac{1}{N}\sum_{i=1}^N |x_i^{k+1}|^2\right) &\le  \frac{1}{c_0}\bar{\phi}_N, 	\\
				\label{eq:stnxg}
				\frac{\tau}{N}\sum_{k=0}^{K-1}\sum_{i=1}^N \left(\frac{x_i^{k+1}-x_i^k}{\tau}\right)^2 &\le 2 \bar{\phi}_N, 	\\
				\label{eq:stnpg}
				\frac{\tau}{N}\sum_{k=0}^{K-1}\sum_{i=1}^N \left|N(\lambda_i^{k+1} - \lambda_{i-1}^{k+1})\right|^2 &\le 4c_2^2T \left(1 + \frac{\bar{\phi}_N}{c_0}\right) + 4\bar{\phi}_N,\text{ and }\\
				\label{eq:stnp}
				\frac{\tau}{N}\sum_{k=0}^{K-1}\sum_{i=1}^N |\lambda_i^{k+1}|^2 &\le 4c_2^2  T \left(1 + \frac{\bar{\phi}_N}{c_0}\right) + 4\bar{\phi}_N.
		\end{align}
	\end{lemma}
	\begin{remark}
		All the bounds in~\Cref{lem:JKO_est} depend on $N$ only through the quantity $\bar{\phi}_N = \frac{1}{N}\sum_{i=1}^N\phi(x_i^0)$. Condition \eqref{x^0_phi} will thus be crucial in the $N\to \infty$ limit.
	\end{remark}
	\begin{proof}
		To prove~\eqref{eq:utsnx}, we fix $k\geq 0$ and take $\mathbf{X}^k = (x_1^k, \dots, x_N^k)$ as a competitor in~\eqref{eq:JKO_micro} to get
		\begin{equation}
			\label{eq:JKO_comp}
			\frac{1}{N}\sum_{i=1}^N\phi(x_i^{k+1}) + \frac{|\mathbf{X}^{k+1}-\mathbf{X}^k|^2}{2N\tau}\le \frac{1}{N}\sum_{i=1}^N\phi(x_i^k).
		\end{equation}
		We easily deduce from~\eqref{eq:JKO_comp} that
		\[
		\frac{1}{N}\sum_{i=1}^N\phi(x_i^{k+1}) \le \frac{1}{N}\sum_{i=1}^N\phi(x_i^0) = \bar{\phi}_N.
		\]
		Using the fact that $\phi$ satisfies~\eqref{phi}, we obtain~\eqref{eq:utsnx} since
		\[
		\frac{1}{N}\sum_{i=1}^N |x_i^{k+1}|^2 \le \frac{1}{c_0N}\sum_{i=1}^N\phi(x_i^{k+1}) - \frac{1}{c_0} \le \frac{1}{c_0}\bar{\phi}_N,\quad \forall k=0,1,\dots, K-1.
		\]
		To prove~\eqref{eq:stnxg}, we return to~\eqref{eq:JKO_comp} and sum over all $k$ from $0$ to $K-1$ to get
		\[
		\frac{\tau}{2N}\sum_{k=0}^{K-1}\left(\frac{|\mathbf{X}^{k+1}-\mathbf{X}^k|}{\tau}\right)^2 \le \frac{1}{N}\sum_{i=1}^N(\phi (x_i^0) - \phi(x_i^{K})).
		\]
		Since  $-\phi(x_i^{K})\le -c_0< 0$ from~\eqref{phi}, we conclude ~\eqref{eq:stnxg}.
		
To prove~\eqref{eq:stnpg}, note that, for any $i=1,\dots, N$ and $k=0,1,\dots, K-1$, the Euler-Lagrange equation and~\eqref{phi} give
		\[
		\left|N(\lambda_i^{k+1}-\lambda_{i-1}^{k+1})\right| =  \left|\phi'(x_i^{k+1}) + \frac{x_i^{k+1}-x_i^k}{\tau}\right| \le c_2 + c_2|x_i^{k+1}| + \frac{|x_i^{k+1}-x_i^k|}{\tau}.
		\]
		Using Young's inequality, multiplying by $\frac{1}{N}$, and summing over $i=1,\dots, N$ gives
		\begin{align*}
			\frac{1}{N}\sum_{i=1}^N \left|N(\lambda_i^{k+1} - \lambda_{i-1}^{k+1})\right|^2 &\le 4c_2^2 + 4c_2^2\left(\frac{1}{N}\sum_{i=1}^N|x_i^{k+1}|^2\right) + \frac{2}{N}\sum_{i=1}^N\left(\frac{x_i^{k+1}-x_i^k}{\tau}\right)^2 	\\
			&\le 4c_2^2 + \frac{4c_2^2}{c_0}\bar{\phi}_N + \frac{2}{N}\sum_{i=1}^N\left(\frac{x_i^{k+1}-x_i^k}{\tau}\right)^2.
		\end{align*}
		In the last inequality, we used~\eqref{eq:utsnx} which was previously established in this proof. Multiplying this inequality by $\tau$, summing over all $k=0,1,\dots, K-1$ (recall $T = \tau K$), and using~\eqref{eq:stnxg} yields~\eqref{eq:stnpg}.
		
		To prove~\eqref{eq:stnp}, whenever $j=1,\dots, N$ and $k\geq 0$, we write the Euler-Lagrange equation~\eqref{eq:EL_micro} with the variable $j$ replacing $i$ to get
		\[
		\lambda_j^{k+1} - \lambda_{j-1}^{k+1} = -\frac{1}{N}\left(\phi'(x_j^{k+1}) + \frac{x_j^{k+1}-x_j^k}{\tau}\right),\quad \lambda_0^{k+1} = 0.
		\]
		For any fixed $i=1,\dots, N$, we can write $\lambda_i^{k+1}$ as the (telescopic) sum of the equation above from $j=1$ to $j=i$ to get
		\[
		\lambda_i^{k+1} = \lambda_i^{k+1} - \lambda_0^{k+1} =  \sum_{j=1}^i\left(\lambda_j^{k+1} - \lambda_{j-1}^{k+1}\right) = -\frac{1}{N}\left(\left(\sum_{j=1}^i \phi'(x_j^{k+1})\right) + \left(\sum_{j=1}^i\frac{x_j^{k+1}-x_j^k}{\tau}\right)\right).
		\]
		Using Young's and the Cauchy-Schwarz inequalities, we obtain
		\begin{align*}
			&\quad \left|\lambda_i^{k+1}\right|^2 \le \frac{2}{N^2}\left(\left(\sum_{j=1}^i \phi'(x_j^{k+1})\right)^2 +  \left(\sum_{j=1}^i\frac{x_j^{k+1}-x_j^k}{\tau}\right)^2 \right) 	\\
			&\le \frac{2i}{N^2}\sum_{j=1}^i \left(\left|\phi'(x_j^{k+1})\right|^2 + \left|\frac{x_j^{k+1}-x_j^k}{\tau}\right|^2 \right) \le \frac{2}{N}\sum_{j=1}^N \left(\left|\phi'(x_j^{k+1})\right|^2 + \left|\frac{x_j^{k+1}-x_j^k}{\tau}\right|^2 \right).
		\end{align*}
This inequality holds for all $i$ and implies:
		\begin{align*}
			&\quad \frac{\tau}{N}\sum_{k=0}^{K-1}\sum_{i=1}^N \left|\lambda_i^{k+1}\right|^2 \le \frac{2\tau}{N}\sum_{k=0}^{K-1}\sum_{i=1}^N\left|\phi'(x_i^{k+1})\right|^2 + \frac{2\tau}{N}\sum_{k=0}^{K-1}\sum_{i=1}^N\left|\frac{x_i^{k+1}-x_i^k}{\tau}\right|^2 	\\
			&\le \frac{2\tau}{N}\left(\sum_{k=0}^{K-1}\sum_{i=1}^Nc_2^2\left(1 + \left|x_i^{k+1}\right| \right)^2\right) + 4\bar{\phi}_N \le 4c_2^2T + \left(4c_2^2\tau \sum_{k=0}^{K-1}\frac{\bar{\phi}_N}{c_0}\right) + 4 \bar{\phi}_N = 4c_2^2 T + \frac{4c_2^2T\bar{\phi}_N}{c_0} + 4\bar{\phi}_N
		\end{align*}
where we used~\eqref{phi},~\eqref{eq:stnxg}, and~\eqref{eq:utsnx}.
%
	\end{proof}
	\subsection{Continuous time limit $\tau \to 0$}
	\label{sec:tau_to_zero}
	We now seek to pass to the limit $\tau\to 0$ and prove the existence part of Proposition \ref{prop:exist}.
	Given 
	\[
	\mathbf{X}^k = (x_1^k,x_2^k, \dots, x_N^k)\in \R^N\quad\text{and}\quad \mathbf{\Lambda}^k = (\lambda_0^k, \dots, \lambda_N^k)\in \R^{N+1}
	\]
	solutions of \eqref{eq:EL_micro} and~\eqref{eq:slack_micro}, we define the following interpolations:
	\[
	\left\{
	\begin{array}{ll}
		\mathbf{X}^\tau(t) := \mathbf{X}^{k+1}, 	&\mathbf{X}^\tau(0) := \mathbf{X}^0,\\
		\tilde{\mathbf{X}}^\tau(t) := \frac{\mathbf{X}^{k+1}-\mathbf{X}^k}{\tau}(t - k\tau) + \mathbf{X}^k, 	&\tilde{\mathbf{X}}^\tau(0) := \mathbf{X}^0,\\
		\mathbf{\Lambda}^\tau(t):= \mathbf{\Lambda}^{k+1}, 
	\end{array}
	\right. \quad t\in I_k:=(k\tau, (k+1)\tau], \quad k=0,1,\dots, K-1.
	\]
	We then have the following bounds:
	\begin{lemma}[Compactness in $\tau$]
		\label{lem:tau_cpct}
		For $\phi$ satisfying \eqref{phi}, there is a constant $C = C(N, \bar{\phi}_N, \, c_0, c_2) = \mathcal{O}(N\bar{\phi}_N)$ such that 
		\begin{equation}
			\label{eq:tau_cpct}
			\left\||\mathbf{X}^\tau|^2 + |\tilde{\mathbf{X}}^\tau|^2 \right\|_{L^\infty(0,T)} + \left\|\tilde{\mathbf{X}}^\tau \right\|_{H^1(0,T)}^2 + \left\| \mathbf{\Lambda}^\tau\right\|_{L^2(0,T)}^2 + \left\|\Delta\lambda^{\tau} \right\|_{L^2(0,T)}^2 \le C \quad\hbox{ for all } \tau>0,
		\end{equation}
where $(\Delta\lambda^{\tau}(t))^2:=\sum_{i=1}^N\left|N(\lambda_i^\tau - \lambda_{i-1}^\tau)\right|^2$ for $t\in I_k$. As a consequence, there exists $\mathbf{X}_N = (x_1,\dots, x_N)\in H^1(0,T; \, \R^N)$ and $\mathbf{\Lambda}_N = (\lambda_1,\dots, \lambda_N) \in L^2\left(0,T; \, \R^N\right)$ such that, up to a subsequence, 
		\[
		\mathbf{X}^\tau\text{ and }\tilde{\mathbf{X}}^\tau \to \mathbf{X}_N \text{ in }L^2(0,T; \, \R^N)\quad \text{and}\quad 
		\left\{
		\begin{array}{c}
			\frac{d\tilde{\mathbf{X}}^\tau}{dt} \rightharpoonup \frac{d\mathbf{X}_N}{dt} 	\\
			\mathbf{\Lambda}^\tau \rightharpoonup \mathbf{\Lambda}_N
		\end{array}
		\right.\text{ in }L^2(0,T; \, \R^N)\text{ as }\tau\to 0.
		\]
	\end{lemma}
	\begin{proof}
		Fix $\tau>0, \, k=0,1,\dots, K-1$, and $t\in I_k$. Immediately from~\eqref{eq:utsnx}, we get
$ \left\|\mathbf{X}^\tau\right\|_{L^\infty(0,T)}<\frac{N}{c_0}\bar{\phi}_N$. 
Since
		\[
		\left|\tilde{\mathbf{X}}^\tau(t)\right| \le \left|\mathbf{X}^{k+1} - \mathbf{X}^k\right|\frac{|t-k\tau|}{\tau} + \left|\mathbf{X}^k\right| \le \left|\mathbf{X}^{k+1}\right| + 2\left|\mathbf{X}^k\right|,
		\]
we also have
		$
		\left|\tilde{\mathbf{X}}^\tau(t)\right|^2\le 6 \left|{\mathbf{X}}^\tau(t)\right|^2 \leq  \frac{6N}{c_0}\bar{\phi}_N.
		$
		Summing up over all $k=0,1,\dots, K-1$ allows to apply~\eqref{eq:stnxg} and deduce
		\begin{align*}
			\left\|\frac{d\tilde{\mathbf{X}}^\tau}{dt} \right\|_{L^2(0,T)}^2 = \sum_{k=0}^{K-1}\int_{k\tau}^{(k+1)\tau}\left|\frac{d\tilde{\mathbf{X}}^\tau}{dt}\right|^2\, \mathrm{d}t = \tau\sum_{k=0}^{K-1} \sum_{i=1}^N\left(\frac{x_i^{k+1}-x_i^k}{\tau}\right)^2 \le 2N\bar{\phi}_N.
		\end{align*}
		Using~\eqref{eq:stnp}, we have
		\[
		\left\|\mathbf{\Lambda}^\tau\right\|_{L^2(0,T)}^2 = \sum_{k=0}^{K-1}\int_{k\tau}^{(k+1)\tau}\left|\mathbf{\Lambda}^\tau(t)\right|^2\, \mathrm{d}t = \tau\sum_{k=0}^{K-1}\left(\sum_{i=1}^N \left|\lambda_i^{k+1}\right|^2\right) \le N\times \left(4c_2^2T\left(1 + \frac{\bar{\phi}_N}{c_0}\right) + 4\bar{\phi}_N\right).
		\]
and the last term in \eqref{eq:tau_cpct} is treated similarly using~\eqref{eq:stnpg} so~\eqref{eq:tau_cpct} is established.
		
		\medskip
		
		All the limits follow immediately from these bounds except for the strong convergence of $\mathbf{X}^\tau$, which is proved by estimating
%
		\begin{align*}
			\left|\mathbf{X}^\tau(t) - \tilde{\mathbf{X}}^\tau(t)\right| = \left|\mathbf{X}^{k+1} - \frac{\mathbf{X}^{k+1}-\mathbf{X}^k}{\tau}(t-k\tau) - \mathbf{X}^k\right| = \frac{\left|\mathbf{X}^{k+1}-\mathbf{X}^k\right|}{\tau}|t - (k+1)\tau| \le \left|\mathbf{X}^{k+1} - \mathbf{X}^k\right|.
		\end{align*}
		Recalling~\eqref{eq:JKO_comp}, we therefore get
		\[
		\left|\mathbf{X}^\tau(t) - \tilde{\mathbf{X}}^\tau(t)\right|^2 \le \tau \sum_{i=1}^N\left(\phi(x_i^k) - \phi(x_i^{k+1})\right) \le \tau \sum_{i=1}^N \left(\phi(x_i^0) - \phi(x_i^{k+1})\right).
		\]
		Using~\eqref{phi} to estimate $-\phi(x_i^{k+1})\le -c_0 <0$, we obtain $\left|\mathbf{X}^\tau(t) - \tilde{\mathbf{X}}^\tau(t)\right| \le \sqrt{\tau N \bar{\phi}_N}$. This estimate is true for any arbitrary $\tau>0$ and $t\in[0,T]$. In light of the strong convergence $\tilde{\mathbf{X}}^\tau \to \mathbf{X}_N$ in $L^2(0,T)$, this also implies $\mathbf{X}^\tau$ converges strongly to $\mathbf{X}_N$ in $L^2(0,T)$.
	\end{proof}
	In~\Cref{lem:JKO_est,lem:tau_cpct}, we did not actually use the second derivative of $\phi$ besides estimating $\phi'$. In order to pass to the limit $\tau\to 0$ at the level of the Euler-Lagrange equation~\eqref{eq:EL_micro} and~\eqref{eq:slack_micro}, the uniform bound on $\phi''$ will play a crucial role.
	\begin{proposition}[Solutions to~\eqref{eq:EL_r}--\eqref{eq:lambdabc}]
		\label{prop:EL_r}
		Assume that $\phi$ satisfies~\eqref{phi}. The limits $\mathbf{X}_N = (x_1,\dots, x_N)\in H^1\left(0,T;\, \R^N \right)$ and $\mathbf{\Lambda}_N = (\lambda_1, \dots, \lambda_N)\in L^2\left(0, T; \, \R^N\right)$ obtained from~\Cref{lem:tau_cpct} solve~\eqref{eq:EL_r}--\eqref{eq:lambdabc}.
	\end{proposition}
	\begin{proof}
		Throughout this proof, every convergence statement as $\tau\to 0$ is understood to take place along a subsequence.
		
		By definition in~\eqref{eq:EL_micro}, we have $\lambda_0^\tau(t) = \lambda_N^\tau(t) = 0$ for every $t\in(0,T]$. Hence, in the limit $\tau\to 0$, we recover the boundary conditions \eqref{eq:lambdabc}, $\lambda_0 = \lambda_N = 0$. We can rewrite the discrete evolution equation in~\eqref{eq:EL_micro} as
		\begin{equation}
			\label{eq:disc_EL}
			\frac{1}{N}\phi'(x_i^\tau(t)) + \frac{1}{N}\frac{dx_i^\tau(t)}{dt} + \lambda_i^\tau(t) - \lambda_{i-1}^\tau(t) = 0, \quad \forall i=1,\dots, N, \quad \text{almost every }t\in (0,T).
		\end{equation}
\Cref{lem:tau_cpct} gives $\frac{d\tilde{x}_i^\tau}{dt}\rightharpoonup \frac{dx_i}{dt}$ and $\lambda_i^\tau \rightharpoonup \lambda_i$ weakly in $L^2(0,T)$ as $\tau \to 0$ for every $i=1,\dots, N$. 
For the nonlinear term, we can use~\eqref{phi} to get
		\begin{align*}
			\left\| \phi'(x_i^\tau) - \phi'(x_i) \right\|_{L^2(0,T)}^2 &  \le c_2\|x_i^\tau(t) - x_i(t)\|_{L^2(0,T)}^2
			\end{align*}
			and so the strong convergence $\mathbf{X}^\tau \to \mathbf{X}_N$ from~\Cref{lem:tau_cpct}, 
			implies  $\phi'(x_i^\tau) \to \phi'(x_i)$ strongly in $L^2(0,T)$ as $\tau \to 0$ for every $i=1,\dots, N$. Therefore, we can pass to the limit $\tau\to 0$ in~\eqref{eq:disc_EL} and we get~\eqref{eq:EL_r} which reads
		\[
		\frac{1}{N}\phi'(x_i) + \frac{1}{N}\frac{dx_i}{dt} + \lambda_i - \lambda_{i-1} = 0, \quad \forall i=1,\dots, N, \quad \text{a.e. }t\in[0,T].
		\]
		Finally,  at the $\tau>0$ level, we use the piecewise constant interpolant $\mathbf{X}^\tau$ to express~\eqref{eq:slack_micro} as
		\[
		\lambda_i^\tau \left(-x_{i+1}^\tau + x_i^\tau + \frac{1}{N}\right) = 0\quad\text{and}\quad x_{i+1}^\tau - x_i^\tau \geq  \frac{1}{N}, \quad \forall i=1,\dots, N-1.
		\]
		In the topology of $L^2(0,T)$, the left-hand side of the equality is a product of $\lambda_i^\tau$, which weakly converges, and $-x_{i+1}^\tau + x_i^\tau+\frac{1}{N+1}$, which strongly converges. Passing to the limit $\tau\to 0$ recovers~\eqref{non_overlap} and \eqref{eq:slack_r}.
	\end{proof}
We end this section with some qualitative results about $\mathbf{X}_N(t)$. We note that the embedding $H^1(0,T)\subset C^\frac{1}{2}(0,T)$ implies that
$\mathbf{X}_N(t)$ is an absolutely continuous solution to
	\begin{equation}
		\label{eq:Maury}
		\frac{d\mathbf{X}_N(t)}{dt} + C(\mathbf{X}_N(t))^\circ \ni - \Phi'(\mathbf{X}(t)),\quad t\in[0,T], \quad \mathbf{X}_N(0) = \mathbf{X}^0 \in \mathcal{K}_N,
	\end{equation}
	where, for $\mathbf{X} = (x_1,\dots, x_N)\in \mathcal{K}_N$, the polar cone $C(\mathbf{X})^\circ \subset \R^N$ is given~\cite{MV11} by
	\[
	C(\mathbf{X})^\circ = \left\{
	\left(\mu_1 - \mu_0, \mu_2 - \mu_1, \dots, \mu_N - \mu_{N-1} \right) \, \left| \, 
	\begin{array}{l}
		\mu_i \ge 0, \, \forall i=0,1,\dots, N,\text{ and }	\\
		-x_{j+1}+x_j+\frac 1 N < 0 \implies \mu_j = 0
	\end{array}
	\right.
	\right\}.
	\]
In fact we have:
	\begin{proposition}[Connection to~\cite{MV11}]
		\label{prop:Maury}
		Let $\mathbf{X}_N\in H^1(0,T; \, \R^N)$ and $\mathbf{\Lambda}_N\in L^2(0,T; \, \R^{N+1})$ be described as in~\Cref{prop:EL_r}. We have the following statements.
		\begin{enumerate}
			\item  $\mathbf{X}_N$ is the \textbf{only} absolutely continuous (and, hence $H^1$) solution to~\eqref{eq:Maury}.
			\item $(\mathbf{X}^\tau, \mathbf{\Lambda}^\tau)\to (\mathbf{X}_N,\mathbf{\Lambda}_N)$ holds as in~\Cref{lem:tau_cpct}  for the \textbf{entire} sequence $\tau\to 0$.
			\item  $\mathbf{X}_N$ is a solution of~\eqref{eq:micro}.
		\end{enumerate}
	\end{proposition}
	\begin{proof}
		We alluded to~\cite{MV11} prior to the statement. This proof is based on results in that reference:
		\begin{enumerate}
			\item Theorem 2.10 from~\cite{MV11} asserts that~\eqref{eq:Maury} has exactly one absolutely continuous solution on $[0,T]$. Having shown in~\Cref{prop:EL_r} that $\mathbf{X}_N \in H^1(0,T)\subset C^\frac{1}{2}(0,T)$ is a solution to~\eqref{eq:EL_r}, we conclude uniqueness.
			\item If $(\mathbf{X}^*,\mathbf{\Lambda}^*)$ is another limit point of $(\mathbf{X}^\tau,\mathbf{\Lambda}^\tau)_{\tau>0}$, then uniqueness guarantees $(\mathbf{X}^*,\mathbf{\Lambda}^*) = (\mathbf{X}_N, \mathbf{\Lambda}_N)$.
			\item Remark 2.11 from~\cite{MV11} asserts that the solution of~\eqref{eq:Maury} is also a solution of~\eqref{eq:micro}.
		\end{enumerate}
	\end{proof}
	We wish to repeat that, although~\cite{MV11} already has a well-posedness theory for~\eqref{eq:Maury} (equivalently~\eqref{eq:EL_r}), their method does not provide much information on the element(s) of $C(\mathbf{X}_N(t))^\circ$. Recall~\eqref{eq:tau_cpct} which bounds $x_i$ and $\lambda_i$ in various (discrete) norms. Such estimates are not immediately clear from the theory in~\cite{MV11}. Furthermore, these estimates will be crucial in the $N\to \infty$ limit in~\Cref{sec:N_to_inf}.

	\section{Infinite number of particles $N\to \infty$}
	\label{sec:N_to_inf}
In this section, we show how to pass to the limit of a large number of particles $N\to\infty$ when $2r=\frac 1 N$:
We will show that, up to a subsequence, the discrete (in space) functions $X_N(t,s)$ and $\Lambda_N(t,s)$ converge to $(X(t,s),\Lambda(t,s))$ which we prove solve~\eqref{ode:1} and~\eqref{ode:2}.
This will (almost) prove \Cref{thm:q1}: The only missing piece will be the convergence of the entire sequence, which will be a consequence of the uniqueness result proved in~\Cref{sec:uniqueness}.

	\subsection{Discrete-in-$N$ approximation}
	\label{sec:construct}

For $N\in\mathbb N$, let  $\mathbf{X}_N^0 = (x_1^0,\dots x_N^0) \in \R^N$, and the corresponding solutions
$\mathbf{X}_N\in H^1(0,T; \, \R^N)$, $\mathbf{\Lambda}_N \in L^2\left(0,T; \, \R^{N+1}\right)$  of \eqref{eq:EL_r}--\eqref{eq:lambdabc} be as constructed  in \Cref{sec:tau_to_zero}.

Then the interpolations, $X_N(t,s)$ by \eqref{eq:XN} and $\tilde \Lambda_N (t)$ by \eqref{eq:LambdaN}, satisfy

	\begin{equation}\label{eq:EL_sigma}
		\left\{
		\begin{array}{ll}
			\partial_tX_N = -\phi'(X_N) -   \partial_s\tilde{\Lambda}_N, &\text{a.e. } (t,s)\in[0,T]\times \left(0, \, 1\right),	\\
			\tilde{\Lambda}_N\left(t,0\right) = \tilde{\Lambda}_N(t,1) = 0, 	&\text{a.e. }t\in[0,T],
		\end{array}
		\right.
	\end{equation}
	and 	we will derive \eqref{ode:1} by passing to the limit in this equation.
In order to derive \eqref{ode:2}, we need to reformulate \eqref{eq:slack_r} in a similar way which requires the introduction of the following piecewise \textbf{linear} interpolation of $X$ and piecewise \textbf{constant} interpolation for the discrete pressure:
\begin{equation}\label{eq:quantile_piece_lin}
	\tilde{X}_N(t,s) =
	N(x_{i+1}(t) - x_{i}(t))\left(s - \frac{i}{N}\right) + x_{i}(t),\quad s\in\left[\frac{i}{N}, \frac{i+1}{N}\right]  \quad  i=0,\dots, N-1,
	\end{equation}
recalling $x_0(t):= x_1(t) - \frac{2}{N}$ and
\[
\Lambda_N(t,s):=  
		\lambda_{i}(t), 	\qquad s\in\left[\frac{i}{N}, \frac{i+1}{N}\right) 
	 ,\quad i=0,\dots, N-1.
\]
Since $\pa_s \tilde{X}_N(t,s) = N(x_{i+1}(t) - x_i(t))$ for $s\in\left(\frac{i}{N}, \frac{i+1}{N}\right)$, 
	we can rewrite~\eqref{eq:slack_r} as
	\begin{equation}
		\label{eq:slack_sigma}
		\Lambda_N(t,s)\left(1 - \partial_s\tilde{X}_N(t,s) \right) = 0\quad \text{and}\quad \partial_s\tilde{X}_N(t,s)\ge 1 \quad \text{a.e. }(t,s)\in[0,T]\times[0,1].
	\end{equation}

We now need to pass to the limit $N\to\infty$ in \eqref{eq:EL_sigma} and \eqref{eq:slack_sigma}, making sure that $X_N$ and  $\tilde{X}_N$ (as well as $\Lambda_N$ and $\tilde{\Lambda}_N$) have the same limits.
Note that passing to the limit in the term $\phi'(X_N)$ in \eqref{eq:EL_sigma} will require some strong convergence of $X_N$.

	\subsection{Estimates and compactness}\label{sec:sigma_cpct}
	We recall that the initial condition are assumed to satisfy \eqref{x^0_phi} and \eqref{bv_endpts}, which we can also write as
$$	\bar\phi_N=	 \frac{1}{N}\sum_{i=1}^N \phi(x_i^0)  \le C\text{ and }   |x_N(0)-x_0(0)| \le C, \, \forall N\in\N,
$$
where $C>0$ is some constant independent of $N$.
We first have the following	 estimates. These follow rather directly from  \eqref{eq:tau_cpct}, and thus we omit the proof.
	\begin{proposition}[Uniform-in-$N$ estimates for the interpolants]
		\label{prop:sigma_cpct_1}
		For $\phi$ and  $\mathbf{X}_N^0\in \mathcal{K}_N$ satisfying \eqref{phi} and ~\eqref{x^0_phi}, there exists a universal constant $C>0$ independent of $N$ such that
		\begin{equation}
			\label{eq:tau_cpct_sigma}
			\left\| X_N\right\|_{L^\infty_tL^2_s}+
\left\| \tilde X_N\right\|_{L^\infty_tL^2_s}+\left\|\partial_t\tilde{X}_N\right\|_{L_{t,s}^2} +\left\|\partial_tX_N\right\|_{L_{t,s}^2} + \left\|\tilde{\Lambda}_N\right\|_{L_{t,s}^2} +\left\|\Lambda_N\right\|_{L_{t,s}^2} +  \left\|\partial_s\tilde{\Lambda}_N\right\|_{L_{t,s}^2} \le C.
		\end{equation}
where $L^\infty_t L^2_s=L^\infty(0,T,L^2(0,1))$ and $L_{t,s}^2=L^2(0,T,L^2(0,1))$.
	\end{proposition}
In order to get the strong convergence of $X_N$, we will need the following crucial $BV$ estimate, whose proof is postponed to the end of this subsection.
	\begin{proposition}[Uniform-in-$N$ BV bound for $\tilde{X}_N$]
		\label{prop:BV_est}
		For $\phi$ and $\mathbf{X}_N^0\in \mathcal{K}_N$ satisfying \eqref{phi} and \eqref{bv_endpts},  
there exists a constant $C>0$ independent of $N$ such that
		\begin{equation}
			\label{eq:BV_est}
	\int_0^1|\partial_s\tilde{X}_N(t,s)|\,\mathrm{d}s=x_N(t) - x_0(t) \le Ce^{\|\phi''\|_{L^\infty}t}, \quad \forall t\in[0,T],\, \forall N\in\mathbb{N}.
		\end{equation}
	\end{proposition}
We point out that \eqref{eq:BV_est} is equivalent to a uniform  bound on the diameter of the support of the associated density distribution function. 
As a consequence, we have:
	
	\begin{corollary}[Compactness for $X_N, \tilde{X}_N, \Lambda_N$, and $\tilde{\Lambda}_N$]
		\label{cor:sigma_cpct_1}
		Assume the $\phi$ and $\mathbf{X}_N^0\in\mathcal{K}_N$ fulfill all of~\eqref{phi} and~\eqref{x^0_phi}--\eqref{bv_endpts}, respectively. There exist curves $X\in H_t^1(0,T;\, L_s^2(0,1))$ and $\Lambda \in L_t^2(0,T; \, H_s^1(0,1))$ such that, up to a subsequence,
		\begin{enumerate}
			\item 
			$
		\tilde{X}_N, X_N \to X$ strongly in $C([0,T]; \, L^p(0,1))$ for any $p\in[1,2)$,
			\item {both} $\Lambda_N$ and $\tilde{\Lambda}_N$ converge to $\Lambda$ weakly in $L_{t,s}^2([0,T] \times [0,1])$ and  $\partial_s\tilde{\Lambda}_N\rightharpoonup \partial_s \Lambda$ weakly in $L_{t,s}^2([0,T]\times [0,1])$ as $N\to \infty$. 
		\end{enumerate}
	\end{corollary}

	\begin{proof}[Proof of Corollary \ref{cor:sigma_cpct_1}]
		From Proposition \ref{prop:sigma_cpct_1}
 and ~\Cref{prop:BV_est}, we know that 
		\begin{itemize}
			\item $\tilde{X}_N$ is uniformly bounded in $L^\infty([0,T];\, BV([0,1]))$ and
			\item $\partial_t \tilde{X}_N$ is uniformly bounded in $L_{t,s}^2([0,T]\times [0,1]) = L_t^2(0,T; \, L_s^2(0,1))$.
		\end{itemize}
		Since the embedding $BV([0,1])\subset L^q(0,1)$ for any $q\in[1,\infty)$ is compact, we can apply the Aubin-Lions lemma and get the strong convergence of $\tilde{X}_N(t,s)$ to $X(t,s)$ (up to a subsequence).

To prove the strong convergence of $X_N(t,s)$, we compute (using \eqref{eq:BV_est}):
	\[
	\int_0^1|\tilde{X}_N(t,s) - X_N(t,s)|\,\mathrm{d}s = \frac{1}{2N}\sum_{i=0}^{N-1} |x_{i+1}(t)-x_i(t)| 
	 = \frac{1}{2N} (x_N(t)-x_0(t))
	\le \frac{C}{2N}e^{\|\phi''\|_{L^\infty}t}.
	\]
We  deduce that $\| \tilde X_N - X_N\|_{L^\infty(0,T,L^1((0,1))} \to 0$.
Since both $\tilde X_N$ and $X_N$ are bounded in $L^\infty_tL^2_s$, this strong convergence holds in $L^\infty_tL^p_s$ for all $p\in [1,2)$.

\medskip

Next, the bounds \eqref{eq:tau_cpct_sigma} imply the existence of  
$\Lambda\in L_{t,s}^2([0,T]\times [0,1])$ and $\tilde{\Lambda}\in L_t^2(0,T; \, H_s^1(0,1))$ such that $\Lambda_N \rightharpoonup \Lambda$ weakly in $L_{t,s}^2([0,T]\times[0,1])$ and $\tilde{\Lambda}_N \rightharpoonup \tilde{\Lambda}$ weakly in $L_t^2(0,T;\, H_s^1(0,1))$.
In order to prove that $\tilde{\Lambda} = \Lambda$, for any $s\in\left[\frac{i}{N},\frac{i+1}{N}\right)$ and $i=0,1,\dots, N-1$, we have
	\[
	|\tilde{\Lambda}_N(t,s) - \Lambda_N(t,s)| =\left| N\left(s-\frac{i+1}{N}\right)(\lambda_i(t)-\lambda_{i-1}(t))\right| \leq 
	 |\lambda_i(t) - \lambda_{i-1}(t)|.
	\]
We deduce
	\[
	|	\tilde{\Lambda}_N(t,s) - \Lambda_N(t,s)|^2 \le \frac{1}{N^2}|N(\lambda_i(t) - \lambda_{i-1}(t))|^2 \le \frac{1}{N}\left\{\frac{1}{N}\sum_{i=1}^N|N(\lambda_i(t) - \lambda_{i-1}(t))|^2\right\}.
	\]
	This inequality is true for almost every $s\in[0,1]$ and thus~\eqref{eq:tau_cpct} with~\eqref{x^0_phi} yields
	\begin{equation}
		\label{eq:lambda_close}
		\int_0^T\esssup_{s\in[0,1]}|	\tilde{\Lambda}_N(t,s) - \Lambda_N(t,s)|^2\, \mathrm{d}t \le \frac{1}{N}\int_0^T\left\{\frac{1}{N}\sum_{i=1}^N|N(\lambda_i(t) - \lambda_{i-1}(t))|^2\right\}\, \mathrm{d}t = \mathcal{O}\left(\frac{\bar{\phi}_N}{N}\right) \overset{N\to \infty}{\to}0.
	\end{equation}
	This estimate shows that $\Lambda = \tilde{\Lambda}$ for almost every $(t,s)\in[0,T]\times[0,1]$.
	\end{proof}

It remains to prove \Cref{prop:BV_est}, which is an immediate consequence of the following lemma:
\begin{lemma}
		\label{lem:omega_i_exp}
Let 	$\omega_i(t) := N(x_{i+1}(t)-x_{i}(t))$. Then 
		\[
		\omega_i(t) \le \omega_i(0)e^{\|\phi''\|_{L^\infty}t}, \quad \forall i=0,1,\dots, N-1, \quad \forall t\in[0,T].
		\] 
	\end{lemma}

	\begin{proof}
	Recall the convention $x_0 = x_1 - \frac{2}{N}$, hence $\omega_0(t) = 2 = \omega_0(0)$. Fix any $i=1,\dots, N-1$ and $t_0\ge 0$. As mentioned before the statement of this lemma, \eqref{eq:slack_r} ensures $\omega_i(t)\ge 1$ for every $t\in[0,T]$. In the case that $\omega_i(t_0) = 1$, we are done because $\omega_i(0)\exp(\|\phi''\|_{L^\infty}t) \ge 1 = \omega_i(t_0)$. So we assume that $\omega_i(t_0)>1$. In view of \eqref{eq:tau_cpct}, we know that $t\mapsto \omega_i(t)$ is $C^{1/2}$, so there exists $t_*$ such that
	$\omega_i(t)>1$ for all $t\in(t_*,t_0]$ with either $\omega_i(t_*)=1$ or $t_*=0$.
	 The inequality $\omega_i>1$ is equivalent to $x_{i+1}-x_i>\frac{1}{N}$ on $(t_*,t_0]$, so condition \eqref{eq:slack_r} implies $\lambda_i(t) = 0$. Using~\eqref{eq:EL_r}, we have
		\begin{align*}
			\frac{d\omega_i}{dt} = N\left(\frac{dx_{i+1}}{dt} - \frac{dx_i}{dt}\right) &= -N\left( \phi'(x_{i+1}) - \phi'(x_i) \right) - N^2(\lambda_{i+1}\underbrace{-2\lambda_i}_{=0}+\lambda_{i-1}) \\
			 & = -\underbrace{N(x_{i+1}-x_i)}_{=\omega_i}\phi''(\xi_i) - N^2(\lambda_{i+1}+\lambda_{i-1}),\quad \text{for some }\xi_i \in (x_i, x_{i+1}).
		\end{align*}
and since the discrete pressures $\lambda_{i+1}$ and $\lambda_{i-1}$ are non-negative quantities,  we obtain the differential inequality
		\[
		\frac{d\omega_i}{dt} \le \|\phi''\|_{L^\infty}\,\omega_i, \quad \text{almost everywhere on } (t_*,t_0).
		\]
We deduce that $\omega_i(t_0) \leq \omega_i(t_*) e^{\|\phi''\|_{L^\infty}(t_0-t_*)} \leq  \omega_i(t_*) e^{\|\phi''\|_{L^\infty}t_0}$ with 
either $\omega_i(t_*)=\omega_i(0)$ or $\omega_i(t_*)=1<\omega_i(0)$. The result follows.
	\end{proof}
		\begin{proof}[Proof of~\Cref{prop:BV_est}]
Using \Cref{lem:omega_i_exp} and the fact that $\partial_s\tilde{X}_N(t,s)\geq 0$, we get
	\begin{align*}
		&\quad \int_0^1|\partial_s\tilde{X}_N(t,s)|\,\mathrm{d}s = \int_0^1\partial_s\tilde{X}_N(t,s)\,\mathrm{d}s = x_N(t) - x_0(t) \leq 	
	\left(x_N^0  - x_0^0\right)e^{\|\phi''\|_{L^\infty}t}
	\end{align*}
which is bounded uniformly in $N\in\N$ owing to~\eqref{bv_endpts} and the proof is complete.
	\end{proof}

		\medskip

			\subsection{Limit $N\to \infty$ of~\eqref{eq:EL_r}--\eqref{eq:lambdabc}}
	\label{sec:EL}
We are now ready to pass to the limit $N\to \infty$ and derive \eqref{ode:1} and \eqref{ode:2}.
First we have
	\begin{proposition}
		\label{thm:char_eq}
		The pair $(X,\Lambda)$ from~\Cref{cor:sigma_cpct_1} solves~\eqref{ode:1} with $X(0,s) = X^0(s)$ for a.e. $s\in[0,1]$.
	\end{proposition}
	\begin{proof}
		We   obtain~\eqref{ode:1} by passing to the limit $N\to \infty$ in~\eqref{eq:EL_sigma}. 
We recall that at this point \Cref{cor:sigma_cpct_1} gives convergence as $N\to\infty$  along some subsequence.
		\medskip
		
				In view of~\Cref{cor:sigma_cpct_1}, 
	we can pass to the limit in all the terms in the continuity equation in the sense of distributions (and in weak $L^2$) except for the non linear term $\phi'(X_N)$.
	The growth assumption~\eqref{phi} says that $|\phi'(X)|\le c_2(1 + |X|)$. Since $X \in L_{t,s}^2$, we deduce that $\phi'(X) \in L_{t,s}^2$. Moreover, we have
$$
|\phi'(X_N(t,s)) - \phi'(X(t,s))|  \leq c_2 |X_N(t,s) - X(t,s)|
$$
and so the strong convergence of $X_N$ to $X$ in $L^\infty(0,T;L^1(0,1))$ implies
the strong convergence of $\phi'(X_N)$ to $\phi'(X)$ in $L^\infty(0,T;L^1(0,1))$.

\medskip

 The boundary conditions on $\tilde{\Lambda}_N$ in \eqref{eq:EL_sigma} implies that $\tilde{\Lambda}_N$ is bounded in $L^2(0,T;H^1_0(0,1))$ and so its weak limit  $\Lambda$ also lies in $L^2(0,T;H^1_0(0,1))$ and thus satisfies the boundary conditions.

		\medskip
		
		Recall that $X_N(0,s) = X_N^0(s) \to X^0 $ in $L^1(0,1)$. Since the convergence of $X_N$ holds in $C([0,T]; \, L^1(0,1))$, we recover $X(0,s) =X^0(s)$.
	\end{proof}

	
%

Next, we want to derive \eqref{ode:2} by passing to the limit in \eqref{eq:slack_sigma} which we recall here:
	\[
	\Lambda_N(1 - \partial_s\tilde{X}_N)=0, \text{ a.e. }(t,s)\in[0,T]\times [0,1].
	\]
We note that ~\Cref{cor:sigma_cpct_1} implies that $\Lambda_N \rightharpoonup \Lambda$ weakly in $L_{t,s}^2([0,T]\times [0,1])$ and $\tilde{X}_N \rightharpoonup X$ weakly in $H_t^1(0,T;\, L_s^2(0,1))$. This information alone is clearly not enough to pass to the limit $N\to \infty$ in~\eqref{eq:slack_sigma}. Nevertheless, we can prove:
	\begin{proposition}[$\lim_{N\to \infty}\eqref{eq:slack_r} = \eqref{ode:2}$]
\label{lem:slack}
The pair $(X,\Lambda)$ from~\Cref{cor:sigma_cpct_1} solves~\eqref{ode:2}, i.e.
\[
\Lambda(t,s)(1 - \partial_sX(t,s))=0 \mbox{ a.e. } (t,s)\in[0,T]\times[0,1].
\]
	\end{proposition}
	\begin{proof} 
First, we claim that we can replace $\Lambda_N(t,s)$ by $\tilde \Lambda_N(t,s)$ in \eqref{eq:slack_sigma} by writing
			\begin{equation}\label{jgdf}
		\int_0^T\int_0^1 \tilde{\Lambda}_N(t,s)(1 - \partial_s\tilde{X}_N(t,s))\, \mathrm{d}s\,\mathrm{d}t = \int_0^T\int_0^1 (\tilde{\Lambda}_N(t,s) - \Lambda_N(t,s))(1 - \partial_s\tilde{X}_N(t,s))\,\mathrm{d}s\,\mathrm{d}t = \mathcal{O}\left(\frac{1}{\sqrt{N}}\right).
		\end{equation}
		The first equality is a direct consequence of~\eqref{eq:slack_sigma}, and since 
		\[
		\sup_{N\in\N}\|\partial_s \tilde{X}_N\|_{L_t^\infty(0,T;\, L_s^1(0,1))}<+\infty \text{ and }\|\tilde{\Lambda}_N - \Lambda_N\|_{L_t^2(0,T;\, L_s^\infty(0,1))}^2 = \mathcal{O}\left(\frac{\bar{\phi}_N}{N}\right),
		\]
	we can write:
		\[
			 \left| \int_0^T\int_0^1 (\tilde{\Lambda}_N - \Lambda_N)(1 - \partial_s\tilde{X}_N)\,\mathrm{d}s\,\mathrm{d}t \right|
\le  \|1 - \partial_s\tilde{X}_N\|_{L_t^2L_s^1}  \|\tilde{\Lambda}_N - \Lambda_N\|_{L_t^2L_s^\infty}\le C \sqrt{\frac{\bar{\phi}_N}{N}}.
		\]
Using~\eqref{jgdf} and the fact that $\tilde\Lambda_N(t,0)=\tilde\Lambda_N(t,1)=0$, we can integrate by parts to write:
	\[
		\mathcal{O}\left(\frac{1}{\sqrt{N}}\right)  =\int_0^T\int_0^1 \tilde{\Lambda}_N(1 - \partial_s\tilde{X}_N)\, \mathrm{d}s\,\mathrm{d}t =		 \int_0^T\int_0^1\tilde{\Lambda}_N\ \, \mathrm{d}s \, \mathrm{d}t + \int_0^T\int_0^1 \partial_s\tilde{\Lambda}_N\ \tilde{X}_N\\, \mathrm{d}s\, \mathrm{d}t.
	\]
\Cref{cor:sigma_cpct_1} allows passing to the limit $N\to \infty$ and then integration by parts (using the homogeneous boundary conditions for $\Lambda$) yield
	\[
	0 = \int_0^T\int_0^1 \Lambda(t,s) \, \mathrm{d}s \, \mathrm{d}t + \int_0^T\int_0^1 \partial_s \Lambda \, X \, \mathrm{d}s \, \mathrm{d}t = \int_0^T\int_0^1\Lambda ( 1 - \partial_s X)\, \mathrm{d}s \, \mathrm{d}t.
	\]
	The limits satisfy $\Lambda\ge 0$ and $\partial_s X \ge 1$ (since $\Lambda_N\ge 0$ and $\partial_s \tilde{X}_N \ge 1$) for almost every $(t,s)\in[0,T]\times[0,1]$, and so this integral equality implies~\eqref{ode:2}.
	\end{proof}

\section{Uniqueness for the Lagrangian system \eqref{ode:1}--\eqref{ode:2}}
\label{sec:uniqueness}
\Cref{thm:char_eq} and~\Cref{lem:slack} almost imply \Cref{thm:q1}. The only missing part is the convergence of the entire sequences. This will follow  from the uniqueness part of~\Cref{thm:wp_micro} which we prove below:
\begin{proposition}[Uniqueness principle for \eqref{ode:1}--\eqref{ode:2}]
\label{prop:unique_micro}
Let $\phi$ and $X^0$ satisfy~\eqref{phi} and~\eqref{x^0}, respectively. For $i=1,2$, let $X^i\in H^1(0,T;\,L^2(0,1))$ and $\Lambda^i \in L^2(0,T;\, H^1(0,1))$ be solutions to~\eqref{ode:1} and~\eqref{ode:2} with $X^i(0,s) = X^0(s)$. Then, $X^1 = X^2$ and $\Lambda^1 = \Lambda^2$.
\end{proposition}
\begin{proof} 
	We differentiate in time the $L_s^2$ difference between $X^1$ and $X^2$ to write
\begin{align*}
	\frac{d}{dt}\frac{1}{2}\int_0^1|X^1 - X^2|^2\,\mathrm{d}s &= - \int_0^1(\phi'(X^1)-\phi'(X^2))(X^1-X^2)\, \mathrm{d}s - \int_0^1(\partial_s\Lambda^1 - \partial_s\Lambda^2)(X^1 - X^2)\, \mathrm{d}s 	\\
	&=:I_1 + I_2.
\end{align*}
Beginning with $I_1$, we use~\eqref{phi} and the mean value theorem to estimate
\[
|I_1| \le c_2\int_0^1|X^1-X^2|^2\,\mathrm{d}s.
\]
As for $I_2$, we integrate by parts and use the trace-zero boundary condition of $\Lambda^1$ and $\Lambda^2$ to get
\[
I_2 = \int_0^1(\Lambda^1 - \Lambda^2)(\partial_sX^1 - \partial_s X^2)\, \mathrm{d}s.
\]
Recall that $\Lambda^1\ge 0$ and $\Lambda^2\ge 0$ for every $s\in[0,1]$. The unit interval can be partitioned into two subsets
\[
[0,1] = \left\{
s\in[0,1]\, : \, \Lambda^1(t,s)\Lambda^2(t,s) >0
\right\} \sqcup \left\{
s\in [0,1]\, : \, \Lambda^1(t,s)\Lambda^2(t,s) = 0
\right\}=:A\sqcup B.
\]
By analysing both sets $A$ and $B$, the saturation condition~\eqref{ode:2} gives the sign $I_2\le 0$.
%
Returning to the time derivative of the $L_s^2$ difference between $X^1$ and $X^2$, we see that
\[
\frac{d}{dt}\frac{1}{2}\int_0^1|X^1 - X^2|^2\, \mathrm{d}s = I_1 + I_2 \le c_2\int_0^1|X^1 - X^2|^2 \, \mathrm{d}s.
\]
Gr\"onwall's inequality yields $X^1 \equiv X^2$ which, when substituted back into~\eqref{ode:1}, implies $\Lambda^1 \equiv \Lambda^2$.
\end{proof}

\section{Proof of \Cref{thm:wp_micro}}
	\label{sec:q1}
The first part of  \Cref{thm:wp_micro} follows from what we have already proved:
Given $X^0$ satisfying \eqref{x^0}, we can proceed as in Lemma \ref{lem:sample_rho^0} to construct a sequence of discrete initial conditions $X_N^0$ satisfying \eqref{x^0_phi}, 
\eqref{bv_endpts}, and \eqref{eq:x^0_conv}. \Cref{thm:q1} gives the existence of the unique solution to \eqref{ode:1} and~\eqref{ode:2}.

\medskip

 We thus turn to the second part of the theorem: we recall that $(\rho,p)$ are related to $(X,\Lambda)$ as follows.
\begin{itemize}	
	\item $\rho(t,\mathrm{d}x)$ is defined as the push-forward through $X(t,\cdot)$ of the uniform measure on $[0,1]$ according to~\eqref{eq:push_forward}:
$$		\rho(t,\cdot) := X(t,\cdot)_\# (\mathrm{d}s ).$$
	\item The pressure $p$ is defined in term of $\Lambda$ and $X$ by 
	~\eqref{def:pressure}: For every $t\in[0,T]$, we denote by $S(t,\cdot)$ the inverse to $s\mapsto X(t,s)$. This inverse is defined on $[X(t,0),X(t,1)]$ but can be extended to $\R$ by setting $S(t,x)=0$ for $x<X(t,0)$, $S(t,x)=1$ for $x>X(t,1)$. Owing to~\eqref{ode:2}, $S(t,\cdot)$ is non-decreasing and $1$--Lipschitz uniformly in $t\in[0,T]$. We invert~\eqref{def:pressure} by defining
	\begin{equation}
		\label{eq:macro_press}
	p(t,x):= \Lambda(t,S(t,x)),\quad \forall (t,x)\in[0,T]\times \R.
	\end{equation}
\end{itemize}
\begin{proof}[Proof of~\Cref{thm:wp_micro}]
It only remains to prove that  the pair of functions $(\rho,p)$ defined from $(X,\Lambda)$ by~\eqref{eq:push_forward} and~\eqref{eq:macro_press} is a weak solution 
			to~\eqref{eq:macro} with initial condition $\rho(0,\cdot) = \rho^0:=X^0(\cdot)_\# (\mathrm{d}s ).$ Recall from~\cite{DMM16} that weak solutions to~\eqref{eq:macro} are unique.
\medskip

First,
by the change of variables formula, we know that $\rho(t,\,X(t,s))  = \frac{1}{\partial_s X(t,s)}$, for $s\in[0,1]$. From this and~\eqref{ode:2}, we immediately obtain the saturation condition and density constraint
\[
p(1-\rho) = 0, \quad \text{and}\quad 0 \le \rho \le 1 \quad \text{a.e. }t\in[0,T],x\in\R.
\]

\medskip

Next, we show that $(\rho,p)\in AC([0,T];\,\mathscr{P}_2(\R))\times L^2(0,T;\, H^1(\R))$: It is well-known (c.f.~\cite[Proposition 2.17]{S15}) that \eqref{eq:push_forward} implies
\[
W_2^2(\rho(t_1),\rho(t_2)) = \|X(t_2,\cdot) - X(t_1,\cdot)\|_{L^2(0,1)}^2 \qquad \mbox{for any $0\le t_1<t_2\le T$  .}
\]
Due to Minkowski's integral inequality and the fact that $X\in H^1(0,T;\, L^2(0,1))$, we can write
\[
W_2(\rho(t_1),\rho(t_2))\le \int_{t_1}^{t_2}\underbrace{\|\partial_tX(t,\cdot)\|_{L^2(0,1)}}_{\in L^2(0,T)}\, \mathrm{d}t
\]
which gives the desired regularity for $\rho$.
Turning to the pressure, recall that $\Lambda \in L^2(0,T;\,H^1(\R))$. We use the saturation condition, \eqref{eq:push_forward}, and~\eqref{eq:macro_press} to see
\[
\int_0^T\int_\R |p(t,x)|^2 \, \mathrm{d}x \,\mathrm{d}t = \int_0^T\int_\R |p(t,x)|^2 \rho(t,x) \, \mathrm{d}x \,\mathrm{d}t = \int_0^T\int_0^1|\Lambda(t,s)|^2\, \mathrm{d}s \, \mathrm{d}t <+\infty.
\]
As for the spatial derivative, the chain rule applied to~\eqref{eq:macro_press} yields
\[
\partial_x p(t,x) = \pa_x S(t,x) \partial_s\Lambda(t,S(t,x))
\]
and, using the fact that $0\leq \pa_x S(t,x) \leq 1$, we can write 
\begin{align*}
\int_0^T\int_\R |\partial_x p|^2\,\mathrm{d}x \,\mathrm{d}t &= \int_0^T\int_\R |\pa_x S(t,x)|^2 |\partial_s\Lambda (t,S(t,x))|^2\,\mathrm{d}x\,\mathrm{d}t 	\\
&\le \int_0^T\int_\R \pa_x S(t,x) |\partial_s\Lambda (t,S(t,x))|^2\,\mathrm{d}x\,\mathrm{d}t = \int_0^T\int_0^1|\partial_s\Lambda(t,s)|^2 \,\mathrm{d}s\mathrm{d}t < +\infty.
\end{align*}
Turning the evolution equation, let us fix $\psi \in C_c^\infty(\R)$ and consider, with the help of the dominated convergence theorem,
\[
\frac{d}{dt}\int_\R \psi(x)\rho(t,x)\, \mathrm{d}x = \frac{d}{dt}\int_0^1 \psi(X(t,s))\, \mathrm{d}s = \int_0^1\psi'(X(t,s))\,\partial_tX(t,s)\, \mathrm{d}s.
\]
Using the differential equation~\eqref{ode:1} for $(X,\Lambda)$, integration by parts, and~\eqref{eq:push_forward} with~\eqref{eq:macro_press} recovers~\eqref{eq:wk_test}.
\end{proof}

\section{Convergence in Eulerian approach}
\label{sec:q2}
The goal of this section is to prove~\Cref{thm:q2}. Throughout this section, we fix $\rho^0$ satisfying~\eqref{rho^0}, and let $X^0$ and $\mathbf{X}_N^0$ be as given in ~\Cref{lem:sample_rho^0}.

We have already seen how to construct the density distribution $\rho_N(t,x)$ and $\tilde\rho_N(t,x)$ from $\mathbf{X}_N(t)$ by \eqref{eq:empirical} and \eqref{eq:histr}.
From the $\lambda_i(t)$ constructed in~\Cref{prop:exist}, we define its Eulerian counterpart, namely the pressure variable. Its piecewise constant version is given by
\begin{equation}
	p_N(t,x) 	:= \sum_{i=0}^{N-1}\lambda_i(t)\chi_{[x_i(t),x_{i+1}(t))}(x), \label{eq:disc_press_const2}
\end{equation}
recalling $x_0 = x_1 - \frac{2}{N}$, and the piecewise linear version by
\begin{align}
	\tilde{p}_N(t,x) &:= N(\lambda_i(t) - \lambda_{i-1}(t))\left( x - \left(x_i(t) - \frac{1}{2N}\right) \right) + \lambda_{i-1}(t), \label{eq:disc_press_lin}	\\
	\text{for every }x&\in \left[x_i(t) - \frac{1}{2N}, \, x_i(t) + \frac{1}{2N}\right] \text{ and } i=1,\dots, N. \notag
\end{align}
With these definitions in hand, we begin with the Eulerian analogue of~\eqref{eq:EL_sigma} and~\eqref{eq:slack_sigma}.
\begin{proposition}
\label{prop:disc_wk_macro}
Let $(\mathbf{X}_N,\mathbf{\Lambda}_N)$ be described as in~\Cref{prop:exist} and inherit all the assumptions therein. 
Let $\rho_N(t,x)$ be the empirical distribution \eqref{eq:empirical} and $\tilde\rho_N(t,x)$ be the density \eqref{eq:histr}. 
With $p_N$ defined above, we have the following statements.
\begin{enumerate}
	\item The pair $(\tilde{\rho}_N,p_N)$ satisfies
		\begin{equation}\label{wk_sat_disc} 
	p_N(1 - \tilde{\rho}_N) = 0, \quad p_N\ge 0, \text{ and }\,0 \le \tilde{\rho}_N\le 1 \mbox{ a.e. } (t,s)\in[0,T]\times[0,1]. 
	\end{equation}
	\item   The pair $(\rho_N,p_N)$ satisfies  the following PDE holds in the distributional sense
		\begin{equation}
			\label{eq:wk_test_cty_disc}
			\partial_t\rho_N = \partial_x(\rho_N\partial_x\phi) + \partial_{xx} p_N.
		\end{equation}

\end{enumerate}
\end{proposition}
\begin{proof}
Fix any $i=0,\dots,N-1$ and $t\in[0,T]$. According to
the definition of $\tilde{\rho}_N$ and $p_N$, we have
\[
\tilde{\rho}_N(t,x) = \frac{1}{N(x_{i+1}(t)-x_i(t))}\quad \text{ and }\quad p_N(t,x) = \lambda_i(t), \quad \forall x\in[x_i(t),x_{i+1}(t)).
\]
so \eqref{non_overlap}--\eqref{eq:slack_r} implies the first part of the proposition.

Next, for any test function $\psi \in C_c^\infty(\R)$ and $t\in(0,T)$, we use~\eqref{eq:empirical} to get
\[
\frac{d}{dt}\int_\R \psi(x)\, \rho_N(t,\mathrm{d}x) = \frac{d}{dt}\frac{1}{N}\sum_{i=1}^N\psi(x_i(t)) = \frac{1}{N}\sum_{i=1}^N\psi'(x_i(t))\frac{dx_i}{dt}.
\]
Of course, we now use~\eqref{eq:EL_r} and substitute $\frac{dx_i}{dt}$ to get
\begin{align*}
\frac{d}{dt}\int_\R \psi(x)\, \rho_N(t,\mathrm{d}x) &= - \frac{1}{N}\sum_{i=1}^N\psi'(x_i)\phi'(x_i) - \sum_{i=1}^N\psi'(x_i)(\lambda_i-\lambda_{i-1}) \\
&= -\int_\R \psi'(x) \phi'(x) \, \rho_N(t,\mathrm{d}x) -\underbrace{ \sum_{i=1}^N\psi'(x_i)(\lambda_i-\lambda_{i-1})}_{=:I}.
\end{align*}
Using `summation by parts' with $\lambda_0=\lambda_N \equiv 0$, we reveal the following:
\begin{align*}
&\quad I = \sum_{i=1}^N\psi'(x_i)(\lambda_i-\lambda_{i-1}) = \sum_{i=1}^N\psi'(x_i)\lambda_i - \sum_{i=1}^N\psi'(x_{i+1})\lambda_i = -\sum_{i=1}^N(\psi'(x_{i+1})-\psi'(x_i))\lambda_i 	\\ &= -\sum_{i=1}^N\int_{x_i}^{x_{i+1}}\psi''(x)\lambda_i \, \mathrm{d}x = -\sum_{i=1}^N\int_{x_i}^{x_{i+1}}\psi''(x)p_N(t,x) \, \mathrm{d}x = -\int_\R \psi''(x)p_N(t,x)\,\mathrm{d}x.
\end{align*}
\end{proof}
Although $\tilde{p}_N$ does not appear in~\Cref{prop:disc_wk_macro}, the next result demonstrates (among other things) the higher regularity that this particular interpolation enjoys. Therefore, it plays a significant role in passing to the limit $N\to \infty$ in~\eqref{wk_sat_disc}.
\begin{lemma}[Estimates for the interpolations]
\label{lem:diff_interp}
The following estimates hold uniformly in $N$:
	\begin{align}
		\label{eq:rho_C_half}
	W_2^2(\rho_N(t_1),\,\rho_N(t_2)) + W_2^2(\tilde{\rho}_N(t_1),\,\tilde{\rho}_N(t_2))&\le |t_2-t_1|\,\mathcal{O}(\bar{\phi}_N),\quad \forall 0\le t_1<t_2\le T, 	\\
	\label{eq:p_bdd}
	\|p_N\|_{L_{t,x}^2}^2 + \|\tilde{p}_N\|_{L_{t,x}^2}^2 + \|\partial_x\tilde{p}_N\|_{L_{t,x}^2}^2 &= \mathcal{O}(\bar{\phi}_N), 	\\
	\label{eq:rho_close}
 \sup_{t\in[0,T]}W_1(\rho_N(t),\,\tilde{\rho}_N(t)) &= \mathcal{O}\left(\frac{1}{N}\right), \text{ and}	\\
	\label{eq:p_close}
	\|p_N - \tilde{p}_N\|_{L_t^2(0,T;\, L_x^\infty(\R))}^2 &= \mathcal{O}\left(\frac{\bar{\phi}_N}{N}\right).
\end{align}
\end{lemma}
\begin{proof}
To prove~\eqref{eq:rho_C_half}, fix any $0\le t_1<t_2\le T$.
 Since $\rho_N(t,\cdot) = X_N(t,\cdot)_\#(\mathrm{d}s )$ and $X_N\in H^1(0,T;\, L^2(0,1))$,we appeal to ~\cite[Proposition 2.17]{S15} to write
\[
W_2^2(\rho_N(t_1)\,\rho_N(t_2)) = \|X_N(t_2,\cdot)-X_N(t_1,\cdot)\|_{L^2(0,1)}^2 = \int_0^1\left|\int_{t_1}^{t_2}\partial_tX_N(t,s)\,\mathrm{d}t\right|^2\,\mathrm{d}s.
\]
By the Cauchy-Schwarz inequality and~\eqref{eq:tau_cpct_sigma} (c.f. the proof of~\Cref{prop:sigma_cpct_1}), we get
\[
W_2^2(\rho_N(t_1),\,\rho_N(t_2)) \le \int_0^1 |t_2-t_1|\int_{t_1}^{t_2}|\partial_t X_N(t,s)|^2\,\mathrm{d}t\,\mathrm{d}s\le \|\partial_t X_N\|_{L_{t,s}^2}^2\,|t_2-t_1| = |t_2-t_1|\mathcal{O}(\bar{\phi}_N).
\]
A similar estimate is deduced for $W_2^2(\tilde{\rho}_N(t_1),\tilde{\rho}_N(t_2))$ and we conclude~\eqref{eq:rho_C_half}.

To prove~\eqref{eq:p_bdd}, note that we have
\[
\|p_N\|_{L_{t,s}^2}^2 = \int_0^T\int_\R |p_N|^2 \, \mathrm{d}x \,\mathrm{d}t = \int_0^T \sum_{i=1}^N \int_{x_i}^{x_{i+1}}|\lambda_i(t)|^2\,\mathrm{d}x \, \mathrm{d}t = \int_0^T\sum_{i=1}^N|\lambda_i(t)|^2(x_{i+1}(t)-x_i(t))\,\mathrm{d}t.
\]
Using~\eqref{eq:slack_r} with $2r = \frac{1}{N}$ and~\eqref{eq:tau_cpct}, we can write
\[
\|p_N\|_{L_{t,s}^2}^2 = \int_0^T\frac{1}{N}\sum_{i=1}^N|\lambda_i(t)|^2\, \mathrm{d}t = \mathcal{O}(\bar{\phi}_N).
\]
As for $\tilde{p}_N$ and $\partial_x\tilde{p}_N$, the bounds are calculated in a similar way to what was done in the proof of~\Cref{prop:sigma_cpct_1}.


For the comparison between $\rho_N$ and $\tilde{\rho}_N$, we apply~\Cref{lem:dist_emp_q} and~\eqref{eq:BV_est} at any $t\in[0,T]$ to get
\[
W_1(\rho_N(t),\tilde{\rho}_N(t)) = \frac{1}{2N}\left(x_N(t) - x_0(t)\right) = \mathcal{O}\left(\frac{1}{N}\right).
\]
Lastly, to prove~\eqref{eq:p_close}, fix $i=1,...,N$. We will use the particle radius $r = \frac{1}{2N}$ for notational simplicity.
The interval $[x_i(t),x_{i+1}(t))$ can be partitioned into $I_1 \sqcup I_2\sqcup I_3$, where
\[
I_1:= [x_i(t),x_i(t) + r),\, I_2:= [x_i(t) + r ,x_{i+1}(t) - r ),\text{ and } I_3:= [x_{i+1}(t) - r, x_{i+1}(t)).
\]
Then we have
\[
\tilde{p}_N(t,x) - p_N(t,x) =\left\{
\begin{array}{cl}
N(\lambda_i(t) - \lambda_{i-1}(t))\left(x - (x_i(t) - r)\right), 	&\text{if }x \in I_1 \\
0, 	&\text{if }x \in I_2,\\
N(\lambda_{i+1}(t) - \lambda_i(t))\left(x - (x_{i+1}(t) + r)\right), 	&\text{if }x\in I_3.
\end{array}
\right..
\]
Since the coefficients of $\lambda_i - \lambda_{i-1}$ or $\lambda_{i+1} - \lambda_i$ are bounded by $1$, we conclude that 
\[
|\tilde{p}_N(t,x) - p_N(t,x)|\le |\lambda_i(t) - \lambda_{i-1}(t)| + |\lambda_{i+1}(t) - \lambda_i(t)|.
\]
Using Young's inequality, adding extraneous terms, and recalling $\lambda_0 = \lambda_N = 0$, we deduce
\begin{align*}
|\tilde{p}_N(t,x) - p_N(t,x)|^2 &\le 2|\lambda_i(t) - \lambda_{i-1}(t)|^2 + 2|\lambda_{i+1}(t) - \lambda_i(t)|^2 \le 
 \frac{4}{N^2}\sum_{i=1}^N|N(\lambda_i(t) - \lambda_{i-1}(t))|^2.
\end{align*}
Since this inequality is true for any $x\in[x_i(t),x_{i+1}(t))$ and any $i=1,\dots, N$, we recall~\eqref{eq:tau_cpct} to get
\[
\int_0^T\esssup_{x\in \R}|\tilde{p}_N(t,x) - p_N(t,x)|^2\, \mathrm{d}t \le \frac{4}{N}\left(\int_0^T \frac{1}{N}\sum_{i=1}^N|N(\lambda_i(t) - \lambda_{i-1}(t))|^2 \, \mathrm{d}t\right) = \mathcal{O}\left(\frac{\bar{\phi}_N}{N}\right).
\]
\end{proof}
As we will make precise in the following result, the inequalities~\eqref{eq:rho_C_half} and~\eqref{eq:p_bdd} provide the necessary compactness for the quadruple $(\rho_N,\tilde{\rho}_N,p_N,\tilde{p}_N)$. The estimates~\eqref{eq:rho_close} and~\eqref{eq:p_close} ensure that the limits for $\rho_N$ and $\tilde{\rho}_N$ converge to the same limit and respectively for $p_N$ and $\tilde{p}_N$.
\begin{corollary}
\label{cor:rho_p}
There exists $\rho \in AC([0,T];\,\mathscr{P}_2)$ and $p\in L^2(0,T;\, H^1(\R))$ such that, up to a subsequence, 
\begin{enumerate}[label=(\alph*)]
	\item \textbf{both} $\rho_N$ and $\tilde{\rho}_N$ converge to $\rho$ narrowly uniformly in $t\in[0,T]$ as $N\to \infty$ and
	\item \textbf{both} $p_N$ and $\tilde{p}_N$ converge to $p$ weakly in $L_{t,x}^2([0,T]\times\R)$. Moreover, $\partial_x\tilde{p}_N \rightharpoonup \partial_xp$ weakly in $L_{t,x}^2([0,T]\times \R)$ as $N\to \infty$.
\end{enumerate}
\end{corollary}
\begin{proof}
\begin{enumerate}[label = (\alph*)]
	\item From~\eqref{eq:rho_C_half}, we deduce that $\rho_N$ and $\tilde{\rho}_N$ are uniformly $\frac{1}{2}$-H\"older equicontinuous in time. By a refined version of the Ascoli-Arzel\`a theorem (c.f.~\cite[Proposition 3.3.1]{AGS08}), there exists $\rho$ and $\tilde{\rho}$ such that $\rho_N \to \rho$ and $\tilde{\rho}_N\to \tilde{\rho}$ narrowly uniformly in $t\in[0,T]$. On the other hand, \eqref{eq:rho_close} implies $\rho = \tilde{\rho}$.
	\item From~\eqref{eq:p_bdd}, we deduce that $p_N, \tilde{p}_N$, and $\partial_x \tilde{p}_N$ are uniformly bounded in $L^2([0,T]\times \R)$. Hence, there exists $p\in L^2([0,T]\times \R)$ and $\tilde{p} \in L^2(0,T;\,H^1(\R))$ such that $p_N \rightharpoonup p$ weakly in $L^2([0,T]\times \R)$ and $\tilde{p}_N \rightharpoonup \tilde{p}$ weakly in $L^2(0,T;\,H^1(\R))$. On the other hand, \eqref{eq:p_close} implies that $p = \tilde{p}$.
\end{enumerate}
\end{proof}
We are now in a position to prove~\Cref{thm:q2}.
\begin{proof}[Proof of~\Cref{thm:q2}]
Clearly, the limits $(\rho,p)$ from~\Cref{cor:rho_p} are the candidate weak solutions to~\eqref{eq:macro}. In~\Cref{cor:rho_p}, we only established convergence \textit{along a subsequence}. However, we know from~\cite{DMM16} that any weak solution to~\eqref{eq:macro} is unique. Therefore, all we need to show is that $(\rho,p)$ is a weak solution to~\eqref{eq:macro} in the sense of~\Cref{def:wk_macro}. 
\medskip

First, we need to  pass to the limit $N\to \infty$ in~\eqref{wk_sat_disc} to prove \eqref{wk_sat}. The convergence in~\Cref{cor:rho_p} is enough to confirm $
p\ge0\text{ and }0 \le \rho \le 1$. As for the remaining part of the saturation condition, we will use the same trick as in~\Cref{lem:slack}: Since $(\tilde{\rho}_N,p_N)$ satisfy~\eqref{wk_sat_disc}, we have
\[
\tilde{p}_N(1 - \tilde{\rho}_N) = (\tilde{p}_N - p_N)(1 - \tilde{\rho}_N).
\]
Using H\"older's inequality, the uniform bound $0\le \tilde{\rho}_N\le 1$, and~\eqref{eq:p_close}, we therefore get
\[
\|\tilde{p}_N(1-\tilde{\rho}_N)\|_{L^1([0,T]\times\R)} = \|(\tilde{p}_N-p_N)(1-\tilde{\rho}_N)\|_{L^1([0,T]\times\R)} \le \|1 - \tilde{\rho}_N\|_{L_t^2L_x^1}\|\tilde{p}_N - p_N\|_{L_t^2L_x^\infty} 
 \overset{N\to \infty}{\to}0.
\]
Again, since $0\le \tilde{\rho}_N\le 1$, a density argument allows to deduce that
\[
\tilde{\rho}_N \to \rho\text{ in } L^\infty(0,T;\text{weak-}L^p(\R))\text{ for any }p\in(1,\infty).
\]
Since we know that $\tilde{p}_N\to p$ weakly in $L^2(0,T;\,H^1(\R))$, we have a weak-strong convergence pair and we obtain
\[
\int_0^T\int_\R p(1 - \rho)\,\mathrm{d}x\,\mathrm{d}t = \|p(1-\rho)\|_{L^1([0,T]\times \R)} = \lim_{N\to \infty}\|\tilde{p}_N(1-\tilde{\rho}_N)\|_{L^1([0,T]\times\R)} = 0.
\]
We already know that $p(1-\rho)\ge 0$ for almost every $(t,x)\in[0,T]\times \R$, so we deduce $p(1-\rho) = 0$ and \eqref{wk_sat} follows.

\medskip

Next, we know that $\rho_N \to \rho$ narrowly uniformly in $t\in[0,T]$ by~\eqref{eq:rho_close} and $p_N \rightharpoonup p$ weakly in $L^2([0,T]\times \R)$,
so we can pass to the limit (in the sense of distribution) in \eqref{eq:wk_test_cty_disc}
and 
deduce \eqref{eq:wk_test}.

\medskip

Finally, in order to show that $\rho(0,x) =\rho^0(x)$, let us recall that
\[
\rho_N(0,\cdot) = X_N^0(\cdot)_\#(\mathrm{d}s )\text{ and }\rho^0(\cdot) = X^0(\cdot)_\#(\mathrm{d}s ).
\]
Using~\cite[Proposition 2.17]{S15} and~\eqref{eq:x^0_conv}, we get
\begin{align*}
	&\quad W_1(\rho(0,\cdot), \rho^0(\cdot)) \le \liminf_{N\to\infty}W_1(\rho_N(0,\cdot),\rho^0(\cdot)) = \liminf_{N\to\infty}\|X_N^0 - X^0\|_{L_s^1(0,1)} = 0.
\end{align*}
\end{proof}

\appendix

	\section{Sampling of initial data}
	\label{sec:sampling}
In this section, we briefly discuss how $ X^0$ and $\mathbf{X}_N^0$ can be sampled in terms of an initial probability density $\rho^0$ in the context of~\eqref{eq:macro} so that assumptions~\eqref{x^0} and \eqref{x^0_phi}--\eqref{eq:x^0_conv} are fulfilled.
	\begin{lemma}[Sampling of $\rho^0$]
	\label{lem:sample_rho^0}
	Let $\rho^0$ be a probability density function which 
	satisfies~\eqref{rho^0} and denote 
	\[
	\xi_L := \inf(\supp\rho^0), \quad \xi_R:= \sup(\supp\rho^0).
	\]
	Define $X^0:[0,1]\to \R$ to be the quantile function of $\rho^0$ given by~\eqref{eq:dist} so that $X^0(0) := \xi_L$ and
	\[ 
	X^0(s):=\inf\left\{x\in \R \, \left| \, \int_{-\infty}^x \rho^0(y)\, \mathrm{d}y\right. \ge s\right\},\quad \forall s\in(0,1].
	\]
	Define $\mathbf{X}_N^0=(x_1^0,x_2^0,\dots,x_N^0)$ by setting
	\[
	x_i^0:= X^0\left(\frac{i}{N}\right), \quad \forall i=1,\dots, N.
	\]
	Then, $X^0$ satisfies~\eqref{x^0}, $\mathbf{X}_N^0$ belongs to $\mathcal{K}_N$ and  satisfies~\eqref{x^0_phi}--\eqref{eq:x^0_conv}.
	\end{lemma}
	\begin{proof}
	Since $\rho^0$ is bounded, it will be useful to recall~\eqref{eq:x_rho} which says
	\begin{equation}
		\label{eq:quant_inc}
		\int_{-\infty}^{X^0(s)}\rho^0(y)\,\mathrm{d}y = s, \quad \forall s\in[0,1].
	\end{equation}
By the assumption that $\|\rho^0\|_{L^\infty}\le 1$, for any $s_1,s_2\in[0,1]$ with $s_1<s_2$, we use~\eqref{eq:quant_inc} to deduce
	\[
	X^0(s_2) - X^0(s_1) = \int_{X^0(s_1)}^{X^0(s_2)} \, \mathrm{d}y \ge \int_{X^0(s_1)}^{X^0(s_2)} \rho^0(y) \, \mathrm{d}y = \int_{-\infty}^{X^0(s_2)} \rho^0(y) \, \mathrm{d}y - \int_{-\infty}^{X^0(s_1)} \rho^0(y) \, \mathrm{d}y = s_2-s_1
	\]
which means that $X^0$ satisfies~\eqref{x^0}. 
By definition of $\mathbf{X}_N^0$, this also implies $\mathbf{X}_N^0 \in \mathcal{K}_N$.

	Next, we note that $\mathbf{X}_N^0$ satisfies \eqref{bv_endpts} 
	since $\xi_L^0\leq x_1^0 \leq x_N^0=\xi_R$. 
	
	
	\medskip
	
To show that $\mathbf{X}_N^0$ satisfies~\eqref{x^0_phi}, we first point out that $\int\phi(x) \rho^0(x)\,\mathrm{d}x\in \R$ from~\eqref{phi} and~\eqref{rho^0}, and that~\eqref{eq:quant_inc} implies
$$
\int_{-\infty}^{x_1^0}  \rho^0(x)\, \mathrm{d}x = \int_{\xi_L}^{x_1^0}  \rho^0(x)\, \mathrm{d}x =\frac 1 N 
$$
and 
$$
\int_{x_i^0}^{x_{i+1}^0} \rho^0(x)\, \mathrm{d}x = \frac 1 N \quad \forall i=1,\dots, N-1.
$$
We can thus write
$$
\frac 1 N  \sum_{i=1}^N \phi(x_i^0) =  \int_{\xi_L} ^{x_1^0}  \phi(x_1^0) \rho^0(x)\,\mathrm{d}x 
+		 \sum_{i=1}^{N-1} \int_{x_{i}^0}^{x_{i+1}^0}  \phi(x_{i+1}^0) \rho^0(x)\,\mathrm{d}x
$$
and so
\begin{align}
 \left|	 
	\frac 1 N  \sum_{i=1}^N \phi(x_i^0)-\int_\R\phi(x) \rho^0(x)\,\mathrm{d}x \right|
& =	\left| 
	  \int_{\xi_L} ^{x_1^0 }( \phi(x_1^0)-\phi(x) )\rho^0(x)\,\mathrm{d}x +	 \sum_{i=1}^{N-1} \int_{x_{i}^0}^{x_{i+1}^0} (\phi(x_{i+1}^0)-\phi(x))\rho^0(x)\,\mathrm{d}x \right| \nonumber \\
& \leq \|\phi'\|_{L^\infty}
\left[	  \int_{\xi_L} ^{x_1^0 } (x_1^0-\xi_L )\rho^0(x)\,\mathrm{d}x 
+		 \sum_{i=1}^{N-1} \int_{x_{i}^0}^{x_{i+1}^0}  (x_{i+1}^0-x_i^0) \rho^0(x)\,\mathrm{d}x 
\right]\nonumber \\
&  = \|\phi'\|_{L^\infty}
\frac 1 N \left[	 x_1^0-\xi_L 
+		 \sum_{i=1}^{N-1}  x_{i+1}^0-x_i^0 
\right] \nonumber \\
&= \|\phi'\|_{L^\infty}
\frac 1 N \left|	\xi_R-\xi_L  
\right| .\label{eq:phijhgf}
\end{align}
We deduce~\eqref{x^0_phi} since
$$
\frac 1 N  \sum_{i=1}^N \phi(x_i^0) \leq \int_\R\phi(x) \rho^0(x)\,\mathrm{d}x  + \|\phi'\|_{L^\infty}
\frac 1 N |\xi_R-\xi_L |  \leq C.
$$

	\medskip

Finally, we check~\eqref{eq:x^0_conv}. First we can see that $ X_N^0(s)\geq X^0(s)$ for all $s\in [0,1]$:
		Indeed, the definition of $X_N^0$ (see \eqref{eq:XN}) and the monotonicity of $X^0$ implies that for $s\in(\frac{i-1}{N},\frac iN]$, we have 
		$$ X_N^0(s) = x_i^0 = X^0\left(\frac i N\right) \geq X^0(s).$$			
We can thus write
\begin{align*}
		\|X_N^0 - X^0\|_{L^1(0,1)} = \int_0^1X_N^0(s) - X^0(s) \, \mathrm{d}s 
		& = 	\sum_{i=1}^{N}  \int_{\frac{i-1}{N}}^\frac{i}{N} x_i^0  \, \mathrm{d}s-  \int_0^1  X^0(s) \, \mathrm{d}s = 	\frac 1 N \sum_{i=1}^{N}  x_i^0  -  \int_\R x\rho^0(x)\, \mathrm{d}x 
	\end{align*}
where the last equality follows from the fact that $\rho^0 := X^0(\cdot)_\# (\mathrm{d}s )$.
We can now conclude by using inequality \eqref{eq:phijhgf} with $\phi(x)=x$ (and $\|\phi'\|\leq 1$) to get 
		$$\|X_N^0 - X^0\|_{L^1(0,1)}  \leq \frac 1 N |\xi_R-\xi_L |  \to 0. $$

	\end{proof}

\section{Eulerian/Lagrangian description and Wasserstein metric}
\label{sec:prelim}
We denote $\mathscr{P}(\Omega)$ the space of probability measures on $\Omega$. For $p\in[1,\infty)$, we define the subset $\mathscr{P}_p\subset \mathscr{P}(\R)$ by
\[
\mathscr{P}_p:= \left\{\rho \in \mathscr{P}(\R)\, \left| \,  \int_\R |x|^p \, \rho(\mathrm{d}x) <+\infty \right.\right\}.
\]
We endow $\mathscr{P}_p$ with the $p$-Wasserstein metric
\[
W_p^p(\rho_1,\rho_2)= \inf_{\gamma\in\Gamma(\rho_1,\rho_2)}\int_{\R\times\R}|x-y|^p\,\gamma(\mathrm{d}x,\mathrm{d}y), \quad \forall \rho_1,\rho_2\in\mathscr{P}_p,
\]
where the set of plans between $\rho_1$ and $\rho_2$ is given by
\[
\Gamma(\rho_1,\rho_2)= \left\{
\gamma \in \mathscr{P}(\R\times \R)\, \left| \, \gamma(A \times \R) = \rho_0(A)\text{ and }\gamma(\R \times A) = \rho_1(A)\text{ for every Borel subset }A\subset \R\right.
\right\}.
\]
We refer to the narrow topology on $\mathscr{P}_p$ being in duality with continuous and bounded functions. It is well-known that convergence in the Wasserstein topology implies narrow convergence and convergence in the $p$--moments (c.f.~\cite[Proposition 7.1.5]{AGS08}). We recall, for $p=1$, the Kantorovich-Rubinstein duality formula (c.f.~\cite[Remark 6.5]{V09}) which states
\begin{equation}
	\label{eq:KRdual}
	W_1(\rho_1,\rho_2) = \max \left\{\int_\R \psi(x) \rho_1(\mathrm{d}x) - \left.\int_\R \psi(x)\rho_2(\mathrm{d}x)\,\right| \, \psi:\R\to \R \text{ such that }\mathrm{Lip}(\psi)\le 1
	\right\}.
\end{equation}
We say that $\rho \in AC([0,T]; \, \mathscr{P}_2)$ if and only if there exists some $m\in L^2(0,T)$ such that
\[
W_2(\rho(t_1),\,\rho(t_2))\le \int_{t_1}^{t_2}|m(t)|\,\mathrm{d}t, \quad \forall 0\le t_1<t_2\le T.
\]
We recall that, given $N$ points $x_1<\dots< x_N$ in  $\R$, we can define the empirical measure $\rho_N$ by \eqref{eq:empirical} and the distribution $\tilde \rho_N$ by \eqref{eq:histr} (which requires the addition of a point  $x_0=x_1-\frac 2 N$).
We then have:
	\begin{lemma}
	\label{lem:dist_emp_q}
For any $p\in[1,\infty)$ the $p$-Wasserstein distance between $\rho_N$ and $\tilde{\rho}_N$  given by~\eqref{eq:empirical} and~\eqref{eq:histr}, respectively, is given by
	\[
	W_p^p(\rho_N, \, \tilde{\rho}_N)= \int_0^1|X_N(s) - \tilde{X}_N(s)|^p\, \mathrm{d}s = \frac{1}{(p+1)N}\sum_{i=0}^{N-1}|x_{i+1}-x_i|^p,
	\]
	where $X_N$ and $\tilde{X}_N$ are defined by~\eqref{eq:XN} and~\eqref{eq:quantile_piece_lin}, respectively. In particular, we have 
	\[
	W_1(\rho_N,\tilde{\rho}_N) = \int_0^1|X_N(s) - \tilde{X}_N(s)|\,\mathrm{d}s =  \frac{1}{2N}\left(x_{N}-x_0\right).
	\]
\end{lemma}
\begin{proof}
	Proposition 2.17 from~\cite{S15} asserts the first desired equality
	\[
	W_p^p(\rho_N,\tilde{\rho}_N) = \int_0^1|X_N(s) - \tilde{X}_N(s)|^p\, \mathrm{d}s.
	\]
	As for the second equality, notice that on each piece $s\in\left[\frac{i}{N},\frac{i+1}{N}\right)$ for $i=0,\dots, N-1$, we have
	\[
	|X_N(s) - \tilde{X}_N(s)| = N(x_{i+1}-x_i)\left(s - \frac{i}{N}\right).
	\]
	Hence, we directly calculate
	\begin{align*}
		&\quad \int_0^1|X_N(s) - \tilde{X}_N(s)|^p\, \mathrm{d}s = \sum_{i=0}^{N-1}\int_\frac{i}{N}^\frac{i+1}{N}|X_N(s)-\tilde{X}_N(s)|^p\, \mathrm{d}s = \sum_{i=0}^{N-1} |N(x_{i+1}-x_i)|^p\int_\frac{i}{N}^\frac{i+1}{N}\left|s - \frac{i}{N}\right|^p\, \mathrm{d}s 	\\
		&= \sum_{i=0}^{N-1} |N(x_{i+1}-x_i)|^p \, \frac{1}{(p+1)N^{p+1}} = \frac{1}{(p+1)N}\sum_{i=0}^{N-1} |x_{i+1}-x_i|^p.
	\end{align*}
	Finally, when $p=1$, the sum can be simplified due to its telescopic structure.
\end{proof}

	\bibliographystyle{abbrv}
	\bibliography{refs}
\end{document}